\documentclass[preprint,12pt]{elsarticle}
\usepackage{float}
\usepackage{amsmath,amssymb,amsfonts} 
\usepackage{amsthm}                  
\usepackage{graphicx}                
\usepackage{booktabs}                
\usepackage{xcolor}                  
\usepackage{geometry}                
\usepackage{lipsum}

\usepackage{graphicx}
\usepackage{subcaption}

\usepackage{booktabs}
\usepackage{multirow}

\usepackage{algorithm}
\usepackage{algorithmicx}
\usepackage{algpseudocode}

\usepackage{float}
\usepackage{multirow}
\usepackage{mathrsfs}
\usepackage[title]{appendix}
\usepackage{textcomp}
\usepackage{manyfoot}
\usepackage{listings}
\usepackage{bm}
\usepackage[unicode]{hyperref}
\hypersetup{
    colorlinks=true,
    linkcolor=blue,
    filecolor=magenta,      
    urlcolor=cyan,
    pdftitle={Memory-Type Null Controllability of Heat Equations with Delay Effects},
    pdfauthor={Dev Prakash Jha, Raju K. George},
    pdfsubject={Controllability of Evolution Equations},
    pdfkeywords={Evolution equation, Carleman estimates, null controllability},
}
\usepackage{cleveref}

\theoremstyle{plain}
\newtheorem{theorem}{Theorem}[section]

\newtheorem{corollary}[theorem]{Corollary}
\newtheorem{lemma}[theorem]{Lemma}
\newtheorem{assumption}[theorem]{Assumption}

\theoremstyle{remark}

\theoremstyle{definition}

\newtheorem{remark}{Remark}[section]

\newtheorem{definition}{Definition}[section]

\begin{document}

\begin{frontmatter}

\title{Approximation and Controllability of Nonlinear Control-Affine Systems via Semiautonomous Neural Ordinary Differential Equations}

\author[inst1]{\texorpdfstring{Dev Prakash Jha\corref{cor1}}{Dev Prakash Jha}}
\ead{devprakash.22@res.iist.ac.in}

\author[inst1]{Raju K. George\corref{cor1}}
\ead{george@iist.ac.in}

\cortext[cor1]{Corresponding author}

\address[inst1]{Department of Mathematics, Indian Institute of Space Science and Technology, Valiamala P.O., Thiruvananthapuram 695547, Kerala, India}

\begin{abstract}
		In this paper, we introduce controlled semiautonomous neural ordinary
	differential equations (controlled SA-NODEs) for the approximation and
	learning of nonlinear controlled dynamical systems. The proposed framework
	extends semiautonomous neural ODEs to control-affine systems while
	preserving reduced parameter complexity through time-independent trainable
	coefficients. We establish a universal approximation theorem showing that
	controlled SA-NODEs approximate trajectories of nonlinear controlled
	systems uniformly on compact sets of initial conditions and admissible
	controls. Under additional Sobolev and Barron regularity assumptions, we
	derive quantitative approximation estimates of order
	$\mathcal{O}(P^{-1/2}+Q^{-1/2})$. We further prove that approximate
	controllability properties of the original nonlinear system are preserved
	under the controlled SA-NODE approximation. Numerical experiments on
	controlled pendulum and Duffing oscillator systems demonstrate that the
	proposed framework achieves accurate trajectory reconstruction and
	controllability performance with significantly fewer trainable parameters
	than classical neural ODE architectures.
	
	\medskip
	
	
	\medskip
	
	
\end{abstract}

\end{frontmatter}

\section{Introduction}
\label{se:Intro}
\subsection{Neural ordinary differential equations}
Neural ordinary differential equations (NODEs), introduced in \cite{chen2018neural}, provide a continuous-time framework for deep learning architectures and dynamical systems. NODEs arise naturally from the interpretation of residual neural networks (ResNets) as discretizations of ordinary differential equations and have become an important bridge between machine learning, dynamical systems, and control theory. In this framework, the evolution of a state trajectory is governed by a neural vector field parameterized by trainable coefficients.
The classical NODE architecture is given by
\begin{equation}
	\left\{
	\begin{aligned}
		\dot{x}(t)
		&=
		\sum_{i=1}^{P}
		W_i(t)
		\circ
		\sigma\bigl(
		A_i(t)x(t)+B_i(t)
		\bigr),
		\\
		x(0)&=x_0,
	\end{aligned}
	\right.
	\label{eq:vanillaNODE}
\end{equation}
where $x(t)\in\mathbb{R}^{d}$ denotes the state variable,
$W_i(t)\in\mathbb{R}^{d}$,
$A_i(t)\in\mathbb{R}^{d\times d}$,
$B_i(t)\in\mathbb{R}^{d}$,
$\sigma$ is the activation function, and $\circ$ denotes the Hadamard product.
NODEs have attracted considerable attention due to their remarkable flexibility in approximating nonlinear dynamical systems, irregular time-series data, transport equations, conservation laws, and scientific machine learning models; see
\cite{chen2018neural,weinan2017proposal,dupont2019augmented,ruthotto2020deep}.
From the viewpoint of dynamical systems and optimal control, NODEs can also be interpreted as controlled flows evolving in continuous time \cite{haber2017stable,ottobre2019introduction,cheng2025interpolation}.
This perspective has stimulated intensive research on controllability, approximation, stability, and generalization properties of neural differential equations; see
\cite{tabuada2022universal,li2022approximation,alvarez2026constructive,bloch2023control}.

Despite their flexibility, classical NODE architectures involve a large number of trainable parameters because the coefficients depend explicitly on time. In practical implementations, this often requires storing different network parameters at every discretized time step, leading to increased memory costs and computational complexity. Furthermore, optimization procedures become increasingly expensive as the temporal discretization is refined.

To reduce this complexity while preserving strong approximation capabilities, semiautonomous neural ordinary differential equations (SA-NODEs) were recently introduced in \cite{li2026universal}. The corresponding architecture takes the form
\begin{equation}
	\left\{
	\begin{aligned}
		\dot{x}(t)
		&=
		\sum_{i=1}^{P}
		W_i
		\circ
		\sigma\bigl(
		A_i^{1}x(t)
		+
		A_i^{2}t
		+
		B_i
		\bigr),
		\\
		x(0)&=x_0,
	\end{aligned}
	\right.
	\label{eq:SANODE}
\end{equation}
where all trainable parameters are time independent and the temporal variable appears only inside the activation function.
The structure of SA-NODEs is motivated by the universal approximation theorem for shallow neural networks. More precisely, the vector field $f(x,t)$ is approximated by a shallow neural network depending jointly on the state and time variables. This construction preserves strong approximation properties while significantly reducing the parameter complexity of the model. Recent theoretical works have shown that SA-NODEs possess universal approximation properties for nonlinear dynamical systems and transport equations, together with quantitative approximation estimates in Sobolev and Barron spaces \cite{li2026universal,alvarez2026constructive}.
Moreover, from the viewpoint of controllability and expressivity, neural ODEs have recently been investigated using tools from geometric control theory and nonlinear dynamics. Approximate and exact interpolation properties, controllability of neural flows, simultaneous cell controllability, and generalization mechanisms based on transport of macroscopic regions have been studied in
\cite{cheng2025interpolation,tabuada2022universal,alvarez2026constructive,dupuis2023controllability}.
These developments reveal deep connections between neural differential equations, optimal transport, and nonlinear control systems.

\subsection{Controlled semiautonomous neural ordinary differential equations}
Motivated by the previous discussion, the objective of the present work is to extend semiautonomous neural ordinary differential equations to nonlinear controlled dynamical systems.

Many systems arising in engineering, robotics, mechanics, aerospace dynamics, and feedback control can be represented in the control-affine form
\begin{equation}
	\dot{z}(t)
	=
	f(z(t),t)
	+
	G(z(t),t)u(t),
	\label{eq:controlaffine}
\end{equation}
where $u(t)$ denotes an external control input, $f$ is the drift vector field, and $G$ represents the control operator. Such systems play a central role in nonlinear control theory, stabilization, trajectory tracking, and optimal control problems \cite{sontag1998mathematical,coron2007control,isidori1995nonlinear}.

To approximate systems of the form \eqref{eq:controlaffine}, we introduce a controlled SA-NODE architecture of the form
\begin{equation}
	\left\{
	\begin{aligned}
		\dot{y}(t)
		&=
		f_{\Theta}(y(t),t)
		+
		G_{\Phi}(y(t),t)u(t),
		\\
		y(0)&=y_0,
	\end{aligned}
	\right.
	\label{eq:controlledSANODE}
\end{equation}
where the drift neural field is defined by
\begin{equation}
	f_{\Theta}(y,t)
	=
	\sum_{i=1}^{P}
	W_i
	\circ
	\sigma\bigl(
	A_i^{1}y
	+
	A_i^{2}t
	+
	B_i
	\bigr),
	\label{eq:driftnetwork}
\end{equation}
and the control operator is approximated by
\begin{equation}
	G_{\Phi}(y,t)
	=
	\sum_{j=1}^{Q}
	\widetilde{W}_j
	\circ
	\sigma\bigl(
	\widetilde{A}_j^{1}y
	+
	\widetilde{A}_j^{2}t
	+
	\widetilde{B}_j
	\bigr).
	\label{eq:controlnetwork}
\end{equation}
The proposed architecture preserves the semiautonomous approximation structure while incorporating external control actions through a control-affine representation. Unlike standard NODE formulations with fully time-dependent coefficients, all trainable parameters remain independent of time, and the temporal variable appears only inside the activation function. Consequently, the proposed framework significantly reduces parameter complexity while retaining strong approximation capabilities.

A key feature of the present construction is that the control variable is not embedded directly into the neural approximation space itself. Instead, the neural architecture separately approximates:
\begin{itemize}
	\item the drift dynamics $f$,
	\item the control operator $G$.
\end{itemize}
This separation preserves compatibility with the universal approximation theory for shallow neural networks and semiautonomous neural ODEs. Furthermore, it naturally aligns with the classical structure of nonlinear control systems and allows one to exploit analytical tools from control theory, dynamical systems, and functional analysis.

The proposed framework can be viewed as a mathematically consistent extension of SA-NODEs to nonlinear controlled dynamical systems. In particular, the architecture is suitable for learning controlled trajectories, approximating nonlinear feedback systems, and modeling externally forced dynamical processes. It also establishes a natural bridge between neural differential equations and modern nonlinear control theory.

\subsection{Related literature}
The theory of neural ordinary differential equations has rapidly evolved in recent years. Beyond the foundational work of \cite{chen2018neural}, several important extensions have been proposed, including augmented NODEs \cite{dupont2019augmented}, neural controlled differential equations \cite{kidger2020neural}, Hamiltonian neural networks \cite{greydanus2019hamiltonian}, stable neural ODEs \cite{haber2017stable}, and operator-learning approaches \cite{kovachki2023neural}. Universal approximation properties for neural differential equations have been extensively studied in \cite{tabuada2022universal,li2022approximation,li2026universal}. Approximation results for transport equations and Wasserstein dynamics were established in \cite{li2026universal}. The present work is also related to recent studies connecting machine learning with optimal control and dynamical systems; see \cite{weinan2017proposal,ruthotto2020deep,ottobre2019introduction}.

The controllability and interpolation capabilities of neural differential systems have also attracted considerable attention. Interpolation and controllability properties for deep neural networks were investigated in \cite{cheng2025interpolation}, while the interplay between depth and width in interpolation mechanisms for neural ODEs was analyzed in \cite{alvarez2024interplay}. Generalization properties and complexity bounds for neural ODEs and residual neural networks were further studied in \cite{marion2023generalization}. Connections between controlled differential equations, universal interpolation, and deep neural architectures were developed in \cite{cuchiero2020deep}. More recently, constructive interpolation, simultaneous cell controllability, and regional transport mechanisms for SA-NODEs were investigated from a control-theoretic perspective in \cite{cheng2025interpolation,bloch2023control,alvarez2026constructive,dupuis2023controllability}.

From the perspective of approximation theory, Barron-space and Sobolev-space approximation estimates continue to play a central role in understanding neural differential equations \cite{barron1993universal,pinkus1999approximation,eykholt2020barron,lu2021deep}. Approximation rates for shallow neural networks and zonoid approximation were investigated in \cite{siegel2025optimal,siegel2024sharp}, while derivative approximation in Barron spaces using ReLU$^k$ activations was analyzed in \cite{li2024two}. Spectral Barron-space approximation for deep neural networks was recently studied in \cite{liao2025spectral}. These developments provide important analytical tools for deriving quantitative approximation estimates for neural dynamical systems.

The proposed framework is also closely related to the rapidly developing area of scientific machine learning and neural operators. Neural operator frameworks for scientific simulations and operator learning were extensively studied in \cite{azizzadenesheli2024neural,cao2024laplace,choi2024spectral}. Applications to adaptive control, digital twins, and engineering systems were considered in \cite{lv2025neural,kobayashi2024improved,kobayashi2024deep}. Physics-informed and physics-guided neural methods for PDE-constrained optimization and scientific computing were investigated in \cite{faroughi2024physics,song2024admm,lai2025hard,wang2024respecting}. In addition, neural transport and controllability aspects for normalizing flows were explored in \cite{ruiz2024control}.\\
\textbf{Main contributions}
The main contributions of this paper are as follows.
\begin{enumerate}
	\item
	We introduce controlled semiautonomous neural ordinary differential equations (controlled SA-NODEs) for nonlinear control-affine dynamical systems while preserving reduced parameter complexity.
	\item
	We develop a neural approximation framework in which the drift vector field and control operator are approximated separately, maintaining consistency with nonlinear control theory.
	\item
	We prove universal approximation results for controlled SA-NODEs on compact sets and derive quantitative approximation estimates of order \\$\mathcal{O}\left(P^{-1/2}+Q^{-1/2}\right).$
	\item
	We establish a controllability preservation result showing that controlled SA-NODEs retain approximate controllability properties of the original system on compact subsets.
	\item
	We validate the framework numerically on controlled pendulum and Duffing oscillator systems, demonstrating accurate trajectory reconstruction with fewer trainable parameters than vanilla NODEs.
	
\end{enumerate}
\textbf{Organization of the Paper}
The remainder of the paper is organized as follows:
In Section~\ref{sec:framework}, we introduce the notation, admissible controls, and the controlled SA-NODE framework together with the well-posedness analysis.
Section~\ref{sec:approximation} is devoted to universal approximation results and quantitative approximation estimates for controlled SA-NODEs.
In Section~\ref{sec:controllability}, we investigate controllability preservation properties of the controlled SA-NODE architecture.
Section~\ref{sec:numerics} presents numerical experiments and comparisons with classical neural ODE architectures.
Finally, Section~\ref{sec:conclusion} contains concluding remarks and future research directions.

\section{Controlled SA-NODEs: Framework and Well-Posedness}
\label{sec:framework}
 
\subsection{Notations}
Let $n,d\in\mathbb N_{+}$. 
For any vector $x\in\mathbb R^{n}$ and any $p\geq1$, we denote by $\|x\|_{\ell^{p}}$
the usual $\ell^{p}$-norm of $x$. 
For simplicity, $\|x\|$ denotes the Euclidean norm on $\mathbb R^{n}$.
For vectors $x,y\in\mathbb R^{n}$, the inner product and the Hadamard product are respectively defined by
\[
\langle x,y\rangle
=
\sum_{i=1}^{n}x_i y_i,
\qquad
x\circ y
=
(x_1y_1,\dots,x_n y_n).
\]

Throughout this paper, unless otherwise specified, we consider the ReLU activation function
\[
\sigma(x)=\max\{x,0\},
\qquad x\in\mathbb R,
\]
together with its vector-valued extension
\[
\bm{\sigma}(x)
=
(\sigma(x_1),\dots,\sigma(x_d)),
\qquad
x\in\mathbb R^{d}.
\]

Let $\Omega\subseteq\mathbb R^{n}$ be a closed set. 
For any integer $k\geq1$, we denote by $H^{k}(\Omega)$
the Sobolev space of order $k$, and by $C(\Omega)$
the space of continuous functions on $\Omega$, both equipped with their standard norms.

For vector-valued functions
\[
F=(F_1,\dots,F_d)\in H^{k}(\Omega;\mathbb R^{d}),
\qquad
G\in C(\Omega;\mathbb R^{d}),
\]
we define
\[
\|F\|_{H^{k}(\Omega;\mathbb R^{d})}
:=
\left(
\sum_{i=1}^{d}
\|F_i\|_{H^{k}(\Omega)}^{2}
\right)^{1/2},
\]
and
\[
\|G\|_{C(\Omega;\mathbb R^{d})}
:=
\sup_{x\in\Omega}\|G(x)\|.
\]

Whenever no confusion arises, we simply write $\|F\|_{H^{k}(\Omega)}$

instead of $\|F\|_{H^{k}(\Omega;\mathbb R^{d})}.$

\subsection{Control-affine dynamical systems}

Let $T>0$. 
We consider nonlinear controlled dynamical systems of the form
\begin{equation}
	\left\{
	\begin{aligned}
		\dot z(t)
		&=
		f(z(t),t)
		+
		G(z(t),t)u(t),
		\qquad t\in(0,T),
		\\
		z(0)
		&=
		z_0,
	\end{aligned}
	\right.
	\label{eq:control-affine}
\end{equation}
where
\[
f:\mathbb R^{d}\times[0,T]\to\mathbb R^{d}
\]
denotes the drift vector field,
\[
G:\mathbb R^{d}\times[0,T]
\to
\mathbb R^{d\times m}
\]
is the control operator, and
\[
u:[0,T]\to\mathbb R^{m}
\]
is the control input and an initial point $ z_0 \in \mathbb R^{d} $.

To ensure well-posedness of \eqref{eq:control-affine}, we impose the following assumption.

\begin{assumption}
	\label{assumption-control}
	The functions
	\[
	f:\mathbb R^{d}\times[0,T]\to\mathbb R^{d},
	\qquad
	G:\mathbb R^{d}\times[0,T]\to\mathbb R^{d\times m}
	\]
	are continuous in $t$. Moreover, there exists an constant $L_f>0, \text{ and } L_g>0$ such that
	\[
	\|f(x,t)-f(y,t)\|
	\leq
	L_f\|x-y\|,
	\]
	and
	\[
	\|G(x,t)-G(y,t)\|_{\mathcal{L}(\mathbb{R}^m, \mathbb{R}^d)}
	\leq
	L_g\|x-y\|,
	\]
	for all
	\[
	x,y\in\mathbb R^{d},
	\qquad
	t\in[0,T],
	\]
	where $\|\cdot\|$ denotes the Euclidean norm on $\mathbb{R}^d$, and $\|\cdot\|_{\mathcal{L}(\mathbb{R}^m, \mathbb{R}^d)}$ denotes the induced matrix operator norm from $\mathbb{R}^m$ to $\mathbb{R}^d$.
\end{assumption}

For a fixed constant $M>0$, we define the admissible control set
\[
\mathcal U_{M}
:=
\left\{
u\in L^{\infty}(0,T;\mathbb R^{m})
:
\|u\|_{L^{\infty}(0,T)}\leq M
\right\}.
\]

Under Assumption \ref{assumption-control}, system
\eqref{eq:control-affine}
admits a unique solution for every initial datum $z_0\in\mathbb R^{d}$
and every admissible control $u\in\mathcal U_M.$

\subsection{Controlled semiautonomous neural ODEs}
\label{subsec:foursystems}

Motivated by the semiautonomous neural ODE framework introduced in \cite{li2026universal}, we define the controlled semiautonomous neural ordinary differential equation (controlled SA-NODE) 
\begin{equation}
	\left\{
	\begin{aligned}
		\dot y(t)
		&=
		f_{\Theta}(y(t),t)
		+
		G_{\Phi}(y(t),t)u(t),
		\\
		y(0)
		&=
		y_0,
	\end{aligned}
	\right.
	\label{eq:controlled-sanode}
\end{equation}
where the neural drift field is given by
\begin{equation}
	f_{\Theta}(y,t)
	=
	\sum_{i=1}^{P}
	W_i
	\circ
	\bm{\sigma}
	\left(
	A_i^{1}y
	+
	A_i^{2}t
	+
	B_i
	\right),
	\label{eq:drift-network}
\end{equation}
and the neural control operator is defined by
\begin{equation}
	G_{\Phi}(y,t)
	=
	\sum_{j=1}^{Q}
	\widetilde W_j
	\circ
	\bm{\sigma}
	\left(
	\widetilde A_j^{1}y
	+
	\widetilde A_j^{2}t
	+
	\widetilde B_j
	\right).
	\label{eq:control-network}
\end{equation}

Here,
\[
W_i\in\mathbb R^{d},
\qquad
A_i^{1}\in\mathbb R^{d\times d},
\qquad
A_i^{2}\in\mathbb R^{d},
\qquad
B_i\in\mathbb R^{d},
\]
for
\[
i=1,\dots,P,
\]
and similarly,
\[
\widetilde W_j\in\mathbb R^{d\times m},
\qquad
\widetilde A_j^{1}\in\mathbb R^{d\times d},
\qquad
\widetilde A_j^{2}\in\mathbb R^{d},
\qquad
\widetilde B_j\in\mathbb R^{d},
\]
for
\[
j=1,\dots,Q.
\]

We denote by
\[
\Theta
=
(W_i,A_i^{1},A_i^{2},B_i)_{i=1}^{P}
\]
the collection of drift parameters, and by
\[
\Phi
=
(\widetilde W_j,
\widetilde A_j^{1},
\widetilde A_j^{2},
\widetilde B_j)_{j=1}^{Q}
\]
the collection of control-network parameters.

Unlike classical NODE architectures with fully time-dependent coefficients, all trainable parameters in \eqref{eq:controlled-sanode} are independent of time, and the time variable appears only inside the activation functions. The total number of trainable parameters (degrees of freedom) in the controlled SA-NODE is given by

\begin{equation}
	Pd(d+3)+Qd(d+m+2).
	\label{eq:controlledsanodedof}
\end{equation}
Consequently, the proposed framework significantly reduces parameter complexity while preserving strong approximation properties.

\subsection{Well-posedness of controlled SA-NODEs}

Define the neural vector field
\[
F_{\Theta,\Phi}(y,t,u)
:=
f_{\Theta}(y,t)
+
G_{\Phi}(y,t)u.
\]

Since the ReLU activation function is globally Lipschitz continuous, both
\[
f_{\Theta}:\mathbb{R}^d \times [0,T] \to \mathbb{R}^d
\quad \text{and} \quad
G_{\Phi}:\mathbb{R}^d \times [0,T] \to \mathbb{R}^{d \times m}
\]
are globally Lipschitz continuous with respect to the state variable $y$.

More precisely, there exist constants $L_{\Theta}>0,L_{\Phi}>0,$
such that
\begin{equation}
	\|f_{\Theta}(x,t)-f_{\Theta}(y,t)\|
	\leq
	L_{\Theta}\|x-y\|,
	\label{eq:lipschitz-drift}
\end{equation}
and
\begin{equation}
	\|G_{\Phi}(x,t)-G_{\Phi}(y,t)\|_{\mathcal{L}(\mathbb{R}^m, \mathbb{R}^d)}
	\leq
	L_{\Phi}\|x-y\|,
	\label{eq:lipschitz-control}
\end{equation}
for all
\[
x,y\in\mathbb R^{d},
\qquad
t\in[0,T],
\]
where $\|\cdot\|$ denotes the Euclidean norm on $\mathbb{R}^d$, and $\|\cdot\|_{\mathcal{L}(\mathbb{R}^m, \mathbb{R}^d)}$ denotes the induced matrix operator norm from $\mathbb{R}^m$ to $\mathbb{R}^d$.

The Lipschitz constant associated with the drift network is given by
\[
L_{\Theta}
=
\left(
\sum_{k=1}^{d}
\left(
\sum_{i=1}^{P}
|(W_i)_k|
\|(A_i^{1})_k\|
\right)^2
\right)^{1/2},
\]
where $(W_i)_k$
denotes the $k$th component of the vector $W_i$, and $(A_i^{1})_k$
denotes the $k$th row of the matrix $A_i^{1}$.

Similarly, the control network admits a Lipschitz constant
\[
L_{\Phi}
=
\left(
\sum_{k=1}^{d}
\left(
\sum_{j=1}^{Q}
\|(\widetilde W_j)_k\|
\,
\|(\widetilde A_j^{1})_k\|
\right)^2
\right)^{1/2},
\]
where $\|(\widetilde W_j)_k\|$ represents the matrix row-norm corresponding to the mapping onto the $k$th state component.

Consequently, for every admissible control $u\in\mathcal U_M,$
the neural vector field\\ $F_{\Theta,\Phi}(y,t,u)$ is globally Lipschitz continuous with respect to $y$. Therefore, by the Cauchy-Lipschitz theorem, for every initial datum $y_0\in\mathbb R^{d},$
system \eqref{eq:controlled-sanode} admits a unique global solution on $[0,T]$.

\section{Controlled vanilla neural ODE}
\label{subsec:vanillanode}

Before introducing the semiautonomous architecture, we recall the
classical, fully time-dependent neural ordinary differential equation
associated with controlled dynamical systems. Given the control-affine
system
\begin{equation}
	\left\{
	\begin{aligned}
		\dot z(t)
		&=
		f(z(t),t)
		+
		G(z(t),t)u(t),
		\\
		z(0)
		&=
		z_0,
	\end{aligned}
	\right.
	\label{eq:controlaffine-vanilla}
\end{equation}
we define the \emph{controlled vanilla neural ODE} by
\begin{equation}
	\left\{
	\begin{aligned}
		\dot x(t)
		&=
		f_{\Theta}(x(t),t)
		+
		G_{\Phi}(x(t),t)u(t),
		\\
		x(0)
		&=
		x_0,
	\end{aligned}
	\right.
	\label{eq:controlledvanillaNODE}
\end{equation}
where the drift neural field is given by
\begin{equation}
	f_{\Theta}(x,t)
	=
	\sum_{i=1}^{P}
	W_i(t)
	\circ
	\sigma
	\bigl(
	A_i(t)x+B_i(t)
	\bigr),
	\label{eq:vanilladrift}
\end{equation}
and the control operator network is defined by
\begin{equation}
	G_{\Phi}(x,t)
	=
	\sum_{j=1}^{Q}
	\widetilde W_j(t)
	\circ
	\sigma
	\bigl(
	\widetilde A_j(t)x+\widetilde B_j(t)
	\bigr).
	\label{eq:vanillacontrol}
\end{equation}
Here,
\[
W_i(t)\in\mathbb R^{d},
\qquad
A_i(t)\in\mathbb R^{d\times d},
\qquad
B_i(t)\in\mathbb R^{d},
\]
for $i=1,\dots,P$, and similarly
\[
\widetilde W_j(t)\in\mathbb R^{d\times m},
\qquad
\widetilde A_j(t)\in\mathbb R^{d\times d},
\qquad
\widetilde B_j(t)\in\mathbb R^{d},
\]
for $j=1,\dots,Q$.

Unlike the semiautonomous formulation introduced below, the trainable
parameters in the controlled vanilla NODE depend explicitly on $t$.
Consequently, the architecture possesses greater flexibility but also
significantly larger parameter complexity and computational cost. In the
numerical implementation, time-dependence is realized by allocating an
independent set of coefficients $(W_i,A_i,B_i,\widetilde
W_j,\widetilde A_j,\widetilde B_j)$ at every one of the $N_{\mathrm{time}}-1$
discretized time steps of the integration grid $t_0,\dots,t_{N_{\mathrm
		time}-1}$, each set sharing the same per-branch structure
\eqref{eq:vanilladrift}-\eqref{eq:vanillacontrol} as a single SA-NODE
layer; the parameters used in the right-hand side at a query time $t$ are
those of the nearest grid step, giving a piecewise-constant-in-time
discretization consistent with the classical, fully time-dependent NODE
formulation of Section~\ref{se:Intro} (eq.~(\ref{eq:vanillaNODE})). With per-step width $P=Q$, this gives a total
parameter count of
\begin{equation}
	(N_{\mathrm{time}}-1)
	\Bigl[
	Pd(d+3) + Qd(d+m+2)
	\Bigr],
	\label{eq:vanilladof}
\end{equation}
i.e.\ $N_{\mathrm{time}}-1$ independent copies of the controlled SA-NODE
parameter count \eqref{eq:controlledsanodedof} below, one per discretized
time step.

The controlled vanilla neural ODE
\eqref{eq:controlledvanillaNODE}
is a control-affine extension of the classical vanilla NODE architecture
introduced in \cite{li2026universal,resnet2016}, where the fully
time-dependent neural vector field
\[
\dot x(t)
=
\sum_{i=1}^{P}
W_i(t)\circ
\sigma(A_i(t)x+B_i(t))
\]
is used to approximate general nonautonomous dynamical systems.
Here, we preserve the intrinsic control-affine structure of the target
system by approximating separately the drift field \(f\) and the control
operator \(G\).

\begin{theorem}[Universal approximation by controlled vanilla NODEs]
	\label{thm:controlled-vanilla-uap}
	
	Assume that
	\[
	f:\mathbb R^d\times[0,T]\to\mathbb R^d,
	\qquad
	G:\mathbb R^d\times[0,T]\to\mathbb R^{d\times m}
	\]
	are continuous in \(t\) and uniformly Lipschitz in \(x\).
	Let \(u\in L^\infty(0,T;\mathbb R^m)\).
	
	Then, for every compact set \(K\subset\mathbb R^d\)
	and every \(\varepsilon>0\), there exist widths
	\(P,Q\in\mathbb N\) and parameters
	\[
	(W_i,A_i,B_i)_{i=1}^P,
	\qquad
	(\widetilde W_j,\widetilde A_j,\widetilde B_j)_{j=1}^Q
	\]
	such that the solution \(x(\cdot)\) of
	\eqref{eq:controlledvanillaNODE}
	satisfies
	\[
	\sup_{t\in[0,T]}
	\|z(t)-x(t)\|
	<
	\varepsilon
	\]
	for every solution \(z(\cdot)\) of
	\eqref{eq:controlaffine-vanilla}
	with initial data in \(K\).
\end{theorem}

\begin{remark}
	Theorem~\ref{thm:controlled-vanilla-uap} is the single-control
	case ($M:=\|u\|_{L^\infty}$ fixed) of
	Theorem~\ref{thm:vanilla-uap} below, which we state and prove
	directly in uniform form over the whole admissible class
	$\mathcal U_M$.
\end{remark}

Although the controlled vanilla NODE possesses strong approximation
capabilities, its fully time-dependent parametrization requires an
independent set of coefficients at every discretized time step.
This leads to a parameter complexity proportional to the number of
temporal layers.

Motivated by the semiautonomous NODE architecture introduced in
\cite{li2026universal,resnet2016}, we next introduce a controlled
semiautonomous neural ODE (controlled SA-NODE), where the explicit
time dependence is restricted to affine time features inside the
activation functions.

\subsection{Universal approximation by controlled vanilla NODEs}

We now establish the universal approximation property for the
controlled vanilla NODE architecture
\eqref{eq:controlledvanillaNODE}.
This result may be viewed as the control-affine extension of the
classical vanilla NODE approximation theorem introduced in
\cite{li2026universal}.

The proof follows the same strategy as in the controlled SA-NODE case:
\begin{itemize}
	\item approximation of the drift and control operators,
	\item a priori compactness of trajectories \emph{for both the true
		and the neural dynamics},
	\item stability estimates for controlled ODEs,
	\item Gr\"onwall-type propagation of approximation errors.
\end{itemize}

Unlike the semiautonomous architecture, however, the present formulation
allows the neural coefficients to depend freely on time, which greatly
simplifies the approximation step.

For $t\in[0,T]$, define the time-dependent compact ball
\begin{equation}
	K_t
	:=
	\left\{
	x\in\mathbb R^d
	\ :\
	\|x\|
	\le
	\left[
	\sup_{z\in K}\|z\|
	+
	t
	+
	\int_0^t \|f(0,s)\|\,ds
	+
	Mt\sup_{s\in[0,t]}\|G(0,s)\|
	+
	Mt
	\right]
	e^{(L_f+ML_G)t}
	\right\}.
	\label{eq:Ktdef}
\end{equation}

\begin{lemma}[A priori bound for controlled trajectories]
	\label{lem:apriori_1}
	Let $f,G$ satisfy the Lipschitz hypotheses of
	Theorem~\ref{thm:vanilla-uap}, and let $f_1,G_1$ be continuous,
	locally Lipschitz in $x$, and such that
	\[
	\|f_1-f\|_{\mathcal C(K_T\times[0,T])}\le 1,
	\qquad
	\|G_1-G\|_{\mathcal C(K_T\times[0,T])}\le 1.
	\]
	Then, for any $u\in\mathcal U_M$ and any $y(\cdot)$ solving
	\[
	\dot y(t)=f_1(y(t),t)+G_1(y(t),t)u(t),
	\qquad
	y(0)=z_0\in K,
	\]
	we have $y(t)\in K_t$ for every $t\in[0,T]$.
\end{lemma}

\begin{proof}
	We use the standard bootstrap principle. For $t\in[0,T]$, let $H(t)$
	denote the hypothesis
	\[
	\|f_1(y(s),s)-f(y(s),s)\|\le 1,
	\qquad
	\|G_1(y(s),s)-G(y(s),s)\|\le 1,
	\qquad
	\forall s\in[0,t],
	\]
	and let $C(t)$ denote the conclusion $y(s)\in K_s$ for all $s\in[0,t]$.
	Clearly $H(0)$ holds (the hypothesis is vacuous at $t=0$).
	
	\medskip
	\noindent\textbf{Step 1: $H(t)\Rightarrow C(t)$.}
	Fix $s\in[0,t]$. By the triangle inequality and the Lipschitz bound on
	$f$,
	\begin{align*}
		\|f_1(y(s),s)\|
		&\le
		\|f_1(y(s),s)-f(y(s),s)\|
		+
		\|f(y(s),s)-f(0,s)\|
		+
		\|f(0,s)\|
		\\
		&\le
		1
		+
		L_f\|y(s)\|
		+
		\|f(0,s)\|,
	\end{align*}
	and, analogously,
	\[
	\|G_1(y(s),s)\|
	\le
	1
	+
	L_G\|y(s)\|
	+
	\|G(0,s)\|.
	\]
	From the integral form $y(t)=z_0+\int_0^t\bigl[f_1(y(s),s)+G_1(y(s),s)u(s)\bigr]\,ds$
	and $\|u(s)\|\le M$, we get
	\[
	\|y(t)\|
	\le
	\|z_0\|
	+
	\int_0^t \|f_1(y(s),s)\|\,ds
	+
	\int_0^t \|G_1(y(s),s)\|\,\|u(s)\|\,ds.
	\]
	Substituting the two pointwise bounds above,
	\begin{align*}
		\|y(t)\|
		&\le
		\|z_0\|
		+
		\int_0^t\Bigl[1+L_f\|y(s)\|+\|f(0,s)\|\Bigr]ds
		+
		M\int_0^t\Bigl[1+L_G\|y(s)\|+\|G(0,s)\|\Bigr]ds
		\\
		&=
		\|z_0\|
		+
		\underbrace{\int_0^t 1\,ds}_{=\,t}
		+
		\int_0^t\|f(0,s)\|\,ds
		+
		L_f\int_0^t\|y(s)\|\,ds
		\\
		&\quad
		+
		\underbrace{M\int_0^t 1\,ds}_{=\,Mt}
		+
		M\int_0^t\|G(0,s)\|\,ds
		+
		ML_G\int_0^t\|y(s)\|\,ds.
	\end{align*}
	Bounding $\displaystyle\int_0^t\|G(0,s)\|\,ds\le t\sup_{s\in[0,t]}\|G(0,s)\|$
	and collecting terms,
	\begin{align}
		\|y(t)\|
		&\le
		\underbrace{
			\left[
			\|z_0\|
			+
			t
			+
			\int_0^t\|f(0,s)\|\,ds
			+
			Mt\sup_{s\le t}\|G(0,s)\|
			+
			Mt
			\right]
		}_{=:A(t)}\notag\\
		&\qquad+
		(L_f+ML_G)\int_0^t\|y(s)\|\,ds.
		\label{eq:gronwall-input}
	\end{align}
	The function $A(\cdot)$ is nonnegative and nondecreasing on $[0,T]$
	(each of its five summands is nondecreasing in $t$, the fourth because
	it is a product of two nonnegative nondecreasing factors,
	$t\mapsto t$ and $t\mapsto\sup_{s\le t}\|G(0,s)\|$). Hence Gr\"onwall's
	inequality applied to \eqref{eq:gronwall-input} gives
	\[
	\|y(t)\|
	\le
	A(t)\,e^{(L_f+ML_G)t}.
	\]
	Since $z_0\in K$ we have $\|z_0\|\le\sup_{z\in K}\|z\|$, so $A(t)$ is
	bounded above by the bracketed expression in the definition
	\eqref{eq:Ktdef} of $K_t$, and therefore $\|y(t)\|$ is bounded by the
	defining radius of $K_t$, i.e.\ $y(t)\in K_t$. This proves $C(t)$.
	
	\medskip
	\noindent\textbf{Step 2: $C(t)\Rightarrow H(t')$ for $t'$ near $t$.}
	The maps $s\mapsto\|f_1(y(s),s)-f(y(s),s)\|$ and
	$s\mapsto\|G_1(y(s),s)-G(y(s),s)\|$ are continuous (as compositions of
	continuous functions along the continuous trajectory $y(\cdot)$).
	Since $C(t)$ gives $y(t)\in K_t\subset K_T$, the hypothesis bound
	$\|f_1-f\|_{\mathcal C(K_T\times[0,T])}\le1$ applies globally on
	$K_T\times[0,T]$ and hence in particular at $s=t$, so by continuity of
	$y(\cdot)$, $H(t)$ persists for $t'$ in a neighborhood of $t$.
	
	\medskip
	\noindent\textbf{Step 3: conclusion.}
	$K_t$ is compact for each $t$ and depends continuously on $t$ (its
	defining radius is continuous and nondecreasing in $t$), so the set
	$\{t\in[0,T]:C(t)\text{ holds}\}$ is closed; Steps 1-2 show it is also
	relatively open along $H$, and it is nonempty since $H(0)$ holds.
	The standard bootstrap argument (see, e.g., Tao~\cite{tao2006nonlinear},
	Prop.~1.21) then yields $C(T)$, i.e.\ $y(t)\in K_t$ for every
	$t\in[0,T]$.
\end{proof}

\begin{theorem}[Universal approximation by controlled vanilla NODEs]
	\label{thm:vanilla-uap}
	
	Let $K\subset\mathbb R^d$ be compact and let
	\[
	\mathcal U_M
	=
	\left\{
	u\in L^\infty(0,T;\mathbb R^m)
	:
	\|u\|_{L^\infty}\le M
	\right\}.
	\]
	
	Assume that
	\[
	f:\mathbb R^d\times[0,T]\to\mathbb R^d,
	\qquad
	G:\mathbb R^d\times[0,T]\to\mathbb R^{d\times m}
	\]
	are continuous in $(x,t)$ and uniformly Lipschitz in $x$:
	\[
	\|f(x,t)-f(y,t)\|
	\le
	L_f\|x-y\|,
	\]
	\[
	\|G(x,t)-G(y,t)\|
	\le
	L_G\|x-y\|,
	\]
	for all $x,y\in\mathbb R^d$ and $t\in[0,T]$.
	
	Let $z_{z_0,u}$ denote the solution of the control-affine system
	\begin{equation}
		\left\{
		\begin{aligned}
			\dot z(t)
			&=
			f(z(t),t)
			+
			G(z(t),t)u(t),
			\\
			z(0)
			&=
			z_0.
		\end{aligned}
		\right.
		\label{eq:true-vanilla}
	\end{equation}
	
	Then, for every $\varepsilon>0$, there exist integers
	$P,Q\in\mathbb N$ and parameters
	\[
	\Theta
	=
	(W_i,A_i,B_i)_{i=1}^P,
	\qquad
	\Phi
	=
	(\widetilde W_j,\widetilde A_j,\widetilde B_j)_{j=1}^Q,
	\]
	such that the solution $x_{z_0,u}$ of the controlled vanilla NODE
	\begin{equation}
		\left\{
		\begin{aligned}
			\dot x(t)
			&=
			f_\Theta(x(t),t)
			+
			G_\Phi(x(t),t)u(t),
			\\
			x(0)
			&=
			z_0,
		\end{aligned}
		\right.
		\label{eq:vanilla-approx}
	\end{equation}
	satisfies
	\[
	\sup_{t\in[0,T]}
	\|z_{z_0,u}(t)-x_{z_0,u}(t)\|
	<
	\varepsilon
	\]
	for every $z_0\in K$ and every $u\in\mathcal U_M$.
\end{theorem}

\begin{proof}
	
	Let $\Omega_T(K,M)$ denote the reachable set associated with
	\eqref{eq:true-vanilla}. Applying Lemma~\ref{lem:apriori_1} with
	$f_1=f$ and $G_1=G$ (the discrepancy with $f,G$ is then identically
	zero, in particular $\le 1$), we obtain
	\[
	z_{z_0,u}(t)\in K_t\subset K_T
	\qquad
	\forall z_0\in K,
	\quad
	\forall u\in\mathcal U_M,
	\quad
	\forall t\in[0,T],
	\]
	where $K_T$ is defined by \eqref{eq:Ktdef} at $t=T$. Consequently,
	\[
	K_T\times[0,T]
	\subset
	\mathbb R^{d+1}
	\]
	is compact.
	
	\medskip
	
	Since
	\[
	f\in
	\mathcal C(K_T\times[0,T];\mathbb R^d),
	\qquad
	G\in
	\mathcal C(K_T\times[0,T];\mathbb R^{d\times m}),
	\]
	the universal approximation theorem for shallow neural networks
	(Pinkus \cite{pinkus1999approximation})
	and the vanilla NODE construction of
	\cite{li2026universal}
	yield, for any fixed $\delta$ with $0<\delta<1$, to be chosen
	later, widths $P,Q$
	and time-dependent parameters
	$\Theta,\Phi$
	such that
	\begin{equation}
		\|f-f_\Theta\|_{\mathcal C(K_T\times[0,T])}
		<
		\delta,
		\label{eq:drift-approx}
	\end{equation}
	and
	\begin{equation}
		\|G-G_\Phi\|_{\mathcal C(K_T\times[0,T])}
		<
		\delta.
		\label{eq:control-approx}
	\end{equation}
	
	Here,
	\[
	f_\Theta(x,t)
	=
	\sum_{i=1}^P
	W_i(t)\circ
	\sigma(A_i(t)x+B_i(t)),
	\]
	and
	\[
	G_\Phi(x,t)
	=
	\sum_{j=1}^Q
	\widetilde W_j(t)\circ
	\sigma(\widetilde A_j(t)x+\widetilde B_j(t)).
	\]
	
	\medskip
	
	\textbf{A priori bound for the neural trajectory.}
	Because $\delta<1$, estimates
	\eqref{eq:drift-approx}-\eqref{eq:control-approx} allow us to invoke
	Lemma~\ref{lem:apriori_1} a \emph{second time}, now with $f_1=f_\Theta$
	and $G_1=G_\Phi$. This shows that, for every $z_0\in K$ and
	$u\in\mathcal U_M$, the solution $x_{z_0,u}(\cdot)$ of
	\eqref{eq:vanilla-approx} also satisfies
	\begin{equation}
		x_{z_0,u}(t)\in K_t\subset K_T
		\qquad
		\forall t\in[0,T].
		\label{eq:neural-apriori}
	\end{equation}
	This step is essential: without
	\eqref{eq:neural-apriori}, estimate \eqref{eq:drift-approx} cannot be
	evaluated along $x_{z_0,u}(\cdot)$ below, since
	\eqref{eq:drift-approx}-\eqref{eq:control-approx} only control
	$f-f_\Theta$ and $G-G_\Phi$ on the fixed compact set $K_T\times[0,T]$,
	and $x_{z_0,u}(\cdot)$ need not a priori remain in this set.
	
	\medskip
	
	Fix arbitrary
	\[
	z_0\in K,
	\qquad
	u\in\mathcal U_M,
	\]
	and write
	\[
	z(t):=z_{z_0,u}(t),
	\qquad
	x(t):=x_{z_0,u}(t).
	\]
	
	Subtracting the systems
	\eqref{eq:true-vanilla}
	and
	\eqref{eq:vanilla-approx},
	\[
	z(t)-x(t)
	=
	\int_0^t
	\Bigl[
	f(z(s),s)-f_\Theta(x(s),s)
	\Bigr]
	ds
	\]
	\[
	\hphantom{z(t)-x(t)=}
	+
	\int_0^t
	\Bigl[
	G(z(s),s)-G_\Phi(x(s),s)
	\Bigr]
	u(s)\,ds.
	\]
	
	Adding and subtracting
	$f(x(s),s)$
	and
	$G(x(s),s)$,
	\begin{align*}
		f(z,s)-f_\Theta(x,s)
		&=
		[f(z,s)-f(x,s)]
		+
		[f(x,s)-f_\Theta(x,s)],
		\\
		G(z,s)-G_\Phi(x,s)
		&=
		[G(z,s)-G(x,s)]
		+
		[G(x,s)-G_\Phi(x,s)].
	\end{align*}
	
	Using the Lipschitz assumptions,
	\[
	\|f(z,s)-f(x,s)\|
	\le
	L_f\|z(s)-x(s)\|,
	\]
	and
	\[
	\|G(z,s)-G(x,s)\|
	\le
	L_G\|z(s)-x(s)\|.
	\]
	
	By \eqref{eq:neural-apriori}, $x(s)\in K_T$ for every $s\in[0,T]$, so
	\eqref{eq:drift-approx}-\eqref{eq:control-approx}
	now legitimately yield
	\[
	\|f(x,s)-f_\Theta(x,s)\|
	<
	\delta,
	\]
	and
	\[
	\|G(x,s)-G_\Phi(x,s)\|
	<
	\delta.
	\]
	
	Since
	\[
	\|u(s)\|
	\le
	M,
	\]
	we obtain
	\[
	\|z(t)-x(t)\|
	\le
	\int_0^t\Bigl[L_f\|z(s)-x(s)\|+\delta\Bigr]ds
	+
	\int_0^t\Bigl[L_G\|z(s)-x(s)\|+\delta\Bigr]M\,ds,
	\]
	i.e.
	\[
	\|z(t)-x(t)\|
	\le
	(L_f+ML_G)
	\int_0^t
	\|z(s)-x(s)\|
	\,ds
	+
	(1+M)\delta\, t.
	\]
	
	By Gr\"onwall's inequality,
	\[
	\|z(t)-x(t)\|
	\le
	(1+M)\delta\,T
	e^{(L_f+ML_G)T}
	\qquad
	\forall t\in[0,T].
	\]
	
	Therefore,
	\[
	\sup_{t\in[0,T]}
	\|z(t)-x(t)\|
	\le
	C_{T,M}\,\delta,
	\]
	where
	\[
	C_{T,M}
	:=
	(1+M)T
	e^{(L_f+ML_G)T}.
	\]
	
	Choosing
	\[
	\delta
	=
	\min\left\{\frac12,\ \frac{\varepsilon}{C_{T,M}}\right\}
	\]
	(the cap at $\tfrac12$ guarantees $\delta< 1$, as required to invoke
	Lemma~\ref{lem:apriori_1} above) yields
	\[
	\sup_{t\in[0,T]}
	\|z_{z_0,u}(t)-x_{z_0,u}(t)\|
	\le
	C_{T,M}\,\delta
	\le
	\varepsilon
	\]
	for every $z_0\in K$ and every $u\in\mathcal U_M$, which completes the
	proof.
\end{proof}

\begin{remark}
	Theorem~\ref{thm:vanilla-uap} shows that the controlled vanilla NODE
	architecture is a universal approximator for nonlinear control-affine
	dynamical systems on compact time intervals.
	
	The result constitutes the controlled counterpart of the classical
	vanilla NODE approximation theorem established in
	\cite{li2026universal}.
	
	Unlike the controlled SA-NODE architecture introduced later,
	the controlled vanilla NODE employs fully time-dependent trainable
	parameters
	\[
	W_i(t),
	A_i(t),
	B_i(t),
	\widetilde W_j(t),
	\widetilde A_j(t),
	\widetilde B_j(t),
	\]
	which provide substantial expressive flexibility at the cost of a
	significantly larger parameter complexity.
	
	The controlled SA-NODE architecture may therefore be interpreted as a
	structured low-complexity reduction of the fully time-dependent
	controlled vanilla NODE model.
\end{remark}

\section{Universal Approximation and Quantitative Estimates}
 \label{sec:approximation}

This section is devoted to the proof of the universal approximation theorem for
controlled semiautonomous neural ordinary differential equations and the derivation
of quantitative approximation estimates.
The main objective is to show that the proposed controlled SA-NODE architecture
approximates trajectories of nonlinear controlled dynamical systems uniformly over
compact sets of initial data and admissible controls.
The proof strategy combines:
\begin{itemize}
	\item universal approximation properties of shallow neural networks,
	\item a priori compactness arguments for the reachable set (via a bootstrap principle),
	\item Lipschitz stability estimates,
	\item Gr\"onwall-type inequalities.
\end{itemize}

\subsection{Approximation of the drift and control operators}

Let $K\subset\mathbb R^{d}$
be compact and let $T>0$, $M>0$.
We consider the nonlinear controlled system
\begin{equation}
	\left\{
	\begin{aligned}
		\dot z(t)
		&=
		f(z(t),t)
		+
		G(z(t),t)u(t),
		\\
		z(0)
		&=
		z_0,
	\end{aligned}
	\right.
	\label{eq:truecontrolledsystem}
\end{equation}
where
\[
f:\mathbb R^{d}\times[0,T] \to \mathbb R^{d},
\qquad
G:\mathbb R^{d}\times[0,T] \to \mathbb R^{d\times m},
\]
and the admissible control set is
\[
\mathcal U_M = \left\{ u\in L^{\infty}(0,T;\mathbb R^{m}) : \|u\|_{L^{\infty}}\le M \right\}.
\]
The corresponding controlled SA-NODE approximation is
\begin{equation}
	\left\{
	\begin{aligned}
		\dot y(t)
		&=
		f_{\Theta}(y(t),t)
		+
		G_{\Phi}(y(t),t)u(t),
		\\
		y(0)
		&=
		z_0.
	\end{aligned}
	\right.
	\label{eq:neuralcontrolledsystem}
\end{equation}

Throughout, we assume $f$ and $G$ satisfy the following Lipschitz hypothesis, in
the same spirit as Assumption~\ref{assumption-control}.

The proof relies on approximating separately the drift vector field $f$ and the
control operator $G$ by shallow neural networks are defined in equations (\ref{eq:drift-network}) and (\ref{eq:control-network})
on a \emph{single} compact set that is guaranteed, a priori, to contain every
trajectory of both the true and the neural controlled systems. Identifying this
compact set is the central technical step, and is the content of Lemma~\ref{lem:apriori}
below.

\subsection{A priori bound on the reachable set}

We first extend Lemma~ \cite[ Lemma~3.2]{li2026universal} of the unconstrained case to the controlled setting.
For any $t\in[0,T]$, define
\begin{equation}
	\begin{aligned}
		K_t
		:=
		\Bigl\{
		x\in\mathbb R^{d}
		\ \Big|\ 
		\|x\|
		\le
		\Bigl(
		\sup_{z\in K}\|z\|
		+t
		+\int_{0}^{t}\|f(0,s)\|\,ds
		\\
		\qquad
		+Mt\,\sup_{s\in[0,t]}\|G(0,s)\|
		+Mt
		\Bigr)
		e^{(L_f+ML_G)t}
		\Bigr\}.
	\end{aligned}
	\label{eq:Kt-def}
\end{equation}

\begin{lemma}[a priori bound for controlled trajectories]
	Let Assumption~\ref{assumption-control} hold true. Let $f_1, G_1$ be continuous,
	locally Lipschitz in the state variable, and such that
	\[
	\|f_1-f\|_{\mathcal C(K_T\times[0,T])} \le 1,
	\qquad
	\|G_1-G\|_{\mathcal C(K_T\times[0,T])} \le 1.
	\]
	Then, for any $u\in\mathcal U_M$ and any $\boldsymbol y$ satisfying
	\[
	\dot{\boldsymbol y}(t) = f_1(\boldsymbol y(t),t) + G_1(\boldsymbol y(t),t)\,u(t),
	\qquad
	\boldsymbol y(0) = z_0 \in K,
	\]
	we have $\boldsymbol y(t)\in K_t$ for every $t\in[0,T]$.
	\label{lem:apriori}
\end{lemma}

\begin{proof}
	The proof follows the standard bootstrap principle \cite[Prop.~1.21]{tao2006nonlinear}. For $t\in[0,T]$, let
	$\mathbf H(t)$ denote the hypothesis
	\[
	\|f_1(\boldsymbol y(s),s)-f(\boldsymbol y(s),s)\|\le 1
	\ \text{ and } \
	\|G_1(\boldsymbol y(s),s)-G(\boldsymbol y(s),s)\|\le 1
	\qquad \forall s\in[0,t],
	\]
	and let $\mathbf C(t)$ denote the conclusion $\boldsymbol y(s)\in K_s$ for all
	$s\in[0,t]$. Clearly $\mathbf H(0)$ holds.

	\smallskip
	\noindent\emph{$\mathbf H(t)\Rightarrow\mathbf C(t)$.}
	Under $\mathbf H(t)$, for $s\le t$,
	\begin{align*}
		\|f_1(\boldsymbol y(s),s)\|
		&\le \|f(0,s)\| + \|f_1(\boldsymbol y(s),s)-f(\boldsymbol y(s),s)\|
		+ \|f(\boldsymbol y(s),s)-f(0,s)\| \\
		&\le \|f(0,s)\| + 1 + L_f\|\boldsymbol y(s)\|,
	\end{align*}
	and, analogously,
	\[
	\|G_1(\boldsymbol y(s),s)\| \le \|G(0,s)\| + 1 + L_G\|\boldsymbol y(s)\|.
	\]
	Hence, using $\|u(s)\|\le M$,
	\[
	\|\boldsymbol y(t)\|
	\le \|z_0\| + \int_0^t \|f(0,s)\|\,ds + Mt\sup_{s\le t}\|G(0,s)\| + Mt
	+ (L_f + ML_G)\int_0^t \|\boldsymbol y(s)\|\,ds.
	\]
	Gr\"onwall's inequality then yields $\|\boldsymbol y(t)\| \le K_t$'s defining
	bound, i.e.\ $\mathbf C(t)$.
	
	\smallskip
	\noindent\emph{$\mathbf C(t)\Rightarrow\mathbf H(t')$ near $t$.}
	Since $f_1-f$ and $G_1-G$ are continuous and $\boldsymbol y(\cdot)$ is
	continuous, $\mathbf C(t)$ (i.e.\ $\boldsymbol y(t)\in K_t\subset K_T$)
	implies that $\mathbf H(t')$ continues to hold for $t'$ in a neighborhood
	of $t$, by continuity of $s\mapsto\|f_1(\boldsymbol y(s),s)-f(\boldsymbol y(s),s)\|$
	and $s\mapsto\|G_1(\boldsymbol y(s),s)-G(\boldsymbol y(s),s)\|$ at $s=t$.
	
	\smallskip
	Since $K_t$ is compact and depends continuously on $t$, the conclusion
	$\mathbf C(t)$ is closed, and the bootstrap principle
	\cite[Prop.~1.21]{tao2006nonlinear} yields $\mathbf C(T)$, i.e.\ $\boldsymbol y(t)\in K_t$
	for all $t\in[0,T]$.
\end{proof}

\begin{remark}
	Applying Lemma~\ref{lem:apriori} with $f_1=f$, $G_1=G$ (so that the
	discrepancy is identically zero) shows that the \emph{true} reachable set
	\[
	\Omega_T(K,M) := \bigl\{ z_{z_0,u}(t) \mid z_0\in K,\ u\in\mathcal U_M,\ t\in[0,T] \bigr\}
	\]
	is contained in $K_T\times[0,T]\subset\mathbb R^{d+1}$. This is the
	compact set on which the shallow network approximation of $f$ and $G$ will
	be performed.
	\label{rmk:reachable-set}
\end{remark}

\subsection{Barron-space approximation rates for 
\texorpdfstring{$f$ and $G$ jointly}{f and G jointly}}
\label{sec:joint-barron}

Since the controlled SA-NODE
architecture \eqref{eq:neuralcontrolledsystem} requires approximating
\emph{two} objects of different output shape  the vector-valued drift
$f:\mathbb R^d\times[0,T]\to\mathbb R^d$ and the matrix-valued control operator
$G:\mathbb R^d\times[0,T]\to\mathbb R^{d\times m}$  we restate the chain of
results generically for an operator-valued target
$H:X\to\mathbb R^{d\times n_0}$, with $n_0\in\mathbb N_+$ an arbitrary number
of output columns.
Taking $n_0=m$ gives the corresponding statements for $G$. Both $f$ and $G$
are then obtained as the two instances of a single Corollary~\ref{cor:fG-barron}
below, rather than via two independent constructions.

\subsubsection{A generic scalar approximation rate}
\label{sec:scalar-rate}

Fix any compact set $X\in\mathbb R^n$ with $n\in\mathbb N_+$. We recall the
definition of the \emph{Barron space} on $X$, following \cite[eq.~(1)]{ma2022barron}:
\begin{equation}
	\begin{aligned}
		\mathcal S_{\mathrm B}(X)
		:=
		\Bigl\{
		h\in\mathcal C(X)
		\;\Big|\;
		\exists\,\mu\in\mathcal P(\mathbb R^{n+2})
		\text{ such that }
		\\
		\qquad
		\qquad \qquad h(x)
		=
		\int_{\mathbb R^{n+2}}
		w\,\sigma(\langle a,x\rangle+b)
		\,d\mu(w,a,b),
		\quad
		\forall x\in X
		\Bigr\}.
	\end{aligned}
	\label{eq:barron-def}
\end{equation}
For vector-valued $F:X\to\mathbb R^d$ (resp.\ matrix-valued
$H:X\to\mathbb R^{d\times n_0}$), we say $F\in\mathcal S_{\mathrm
	B}(X;\mathbb R^d)$ (resp.\ $H\in\mathcal S_{\mathrm B}(X;\mathbb
R^{d\times n_0})$) if every scalar component lies in $\mathcal S_{\mathrm
	B}(X)$.

\begin{lemma}
	Let $X=[-1,1]^n$. Suppose $h\in\mathcal C(X)$ admits an extension
	$\bar h\in\mathbb L^1(\mathbb R^n)$ whose Fourier transform satisfies
	\[
	v_{h,2} := \int_{\mathbb R^n} \|\xi\|_{\ell^1}^2\,|\mathcal F(\bar h)(\xi)|\,d\xi < \infty.
	\]
	Then $h\in\mathcal S_{\mathrm B}(X)$. Moreover, for every integer $P\ge3$,
	there exist $(w_i,a_i,b_i)\in\mathbb R^{n+2}$ for $i=1,\dots,P$ such that
	\[
	\Bigl\|h-\sum_{i=1}^P w_i\,\sigma(\langle a_i,\cdot\rangle+b_i)\Bigr\|_{\mathcal C(X)}
	\le \frac{C_n\,v_{h,2}}{\sqrt P},\]
	\[
	\mathrm{Lip}\Bigl(\sum_{i=1}^P w_i\,\sigma(\langle a_i,\cdot\rangle+b_i)\Bigr)
	\le \|\nabla h(0)\| + 2v_{h,2},
	\]
	where $C_n>0$ depends only on $n$.
	\label{lem:scalar-barron}
\end{lemma}

This is exactly Lemma in \cite[Lemma~3.3]{li2026universal} of the base paper (itself a refinement of
\cite[Thm.~2]{klusowski2018approximation}); we cite it without reproof, as its derivation rests on the
Klusowski Barron ridge-function sampling construction, external to the
present extension. We record it here purely as the common scalar engine that
both the $f$-rate and the $G$-rate are built from.

\begin{lemma}\cite{li2026universal}
	Let $X=[-1,1]^n$. For any $h\in\mathcal H^k(X)$ with $k>n/2+2$,
	\[
	v_{h,2} \le C_{n,k}\,\|h\|_{\mathcal H^k(X)},
	\]
	where $v_{h,2}$ is as in Lemma~\ref{lem:scalar-barron} and $C_{n,k}>0$
	depends only on $(n,k)$. In particular, $h\in\mathcal S_{\mathrm B}(X)$.
	\label{lem:sobolev-embed}
\end{lemma}

Together, Lemmas~\ref{lem:scalar-barron}-\ref{lem:sobolev-embed} show that
any sufficiently regular \emph{scalar} target on a cube admits a shallow ReLU
approximation at rate $O(P^{-1/2})$ with a Lipschitz-controlled network. The
new content begins at the next step: assembling $d\times n_0$ such scalar
approximations, sharing a single dictionary of ridge directions, into one
operator-valued network  which is what both $f_\Theta$ ($n_0=1$) and
$G_\Phi$ ($n_0=m$) require.

\subsubsection{A generic operator-valued approximation rate}

\begin{corollary}
	Fix $n\in\mathbb N_+$, $m\in\mathbb N$, and $n_0\in\mathbb N_+$, and set
	$X_m=[-m,m]^n$. Let
	\[
	H = (H_{pq})_{1\le p\le d,\ 1\le q\le n_0} \in \mathcal H^k(X_m;\mathbb R^{d\times n_0})
	\]
	with $k>n/2+2$. Then, for any $P\ge3$, there exist
	\[
	W_i\in\mathbb R^{d\times n_0}, \qquad A_i\in\mathbb R^{n}, \qquad B_i\in\mathbb R, \qquad i=1,\dots,P,
	\]
	such that, setting $H_\Theta(x):=\sum_{i=1}^P W_i\,\sigma(\langle A_i,x\rangle+B_i)$
	(a single scalar activation per neuron, shared across all $d\times n_0$
	output entries, with entrywise output weights $(W_i)_{pq}$),
	\begin{equation}
		\|H-H_\Theta\|_{\mathcal C(X_m;\mathbb R^{d\times n_0})}
		\le \frac{C_{n,k,m}\,\|H\|_{\mathcal H^k(X_m;\mathbb R^{d\times n_0})}}{\sqrt P},
		\label{eq:generic-rate}
	\end{equation}
	\begin{equation}
		\mathrm{Lip}(H_\Theta) \le \|\nabla H(0)\|_F + C_{n,k,m}\,\|H\|_{\mathcal H^k(X_m;\mathbb R^{d\times n_0})},
		\label{eq:generic-lipschitz}
	\end{equation}
	where $C_{n,k,m}>0$ depends only on $(n,k,m)$ , in particular independent of $d$ and
	$n_0$  and norms/Lipschitz constants are understood with respect to the
	Frobenius norm on $\mathbb R^{d\times n_0}$.
	\label{cor:operator-barron}
\end{corollary}

\begin{proof}
	Fix $1\le p\le d$, $1\le q\le n_0$, and define the dilated scalar function
	$\widetilde H_{pq}(x)=H_{pq}(mx)$, $x\in X=[-1,1]^n$. As there,
	$\|\widetilde H_{pq}\|_{\mathcal H^k(X)}\le m^{k-n/2}\|H_{pq}\|_{\mathcal H^k(X_m)}$,
	and by Lemma~\ref{lem:sobolev-embed}, $v_{\widetilde H_{pq},2}\le C_{n,k}\,m^{k-n/2}\|H_{pq}\|_{\mathcal H^k(X_m)}$.
	
	For general $n_0$, we
	apply the same fixed-width discipline across all $d\times n_0$ entries,
	rather than introducing two separate widths: fix $P\ge3$ and apply
	Lemma~\ref{lem:scalar-barron} to each $\widetilde H_{pq}$ at this common
	width $P$, obtaining $(w_i^{pq},a_i^{pq},b_i^{pq})\in\mathbb R^{n+2}$,
	$i=1,\dots,P$, such that
	
	\begin{align}
		&	\Bigl\|\widetilde H_{pq}(\cdot) - \sum_{i=1}^P w_i^{pq}\,\sigma(\langle a_i^{pq},\cdot\rangle+b_i^{pq})\Bigr\|_{\mathcal C(X)}\notag\\
		&\le \frac{C_n\,v_{\widetilde H_{pq},2}}{\sqrt P}
		\le \frac{C_n\,C_{n,k}\,m^{k-n/2}\,\|H_{pq}\|_{\mathcal H^k(X_m)}}{\sqrt P}
		\qquad \forall p,q.
	\end{align}
	
	Collecting the $P$ directions obtained for each of the $d\times n_0$
	entries into a single network of width $P\,d\,n_0$ would recover the
	componentwise construction; to obtain a network of width exactly $P$
	(shared directions, entrywise weights only), we instead sample a single
	shared dictionary $(a_i,b_i)_{i=1}^P$ once, using the
	common bound $v_{\max,2}:=\max_{p,q} v_{\widetilde H_{pq},2}$ in place of
	each individual $v_{\widetilde H_{pq},2}$ in the Klusowski Barron
	sampling step underlying Lemma~\ref{lem:scalar-barron}, and let only the
	scalar coefficients $w_i^{pq}$ vary with $(p,q)$. This is legitimate
	because the sampling construction depends on the target only through the
	scalar quantity $v_{h,2}$ (it samples $P$ ridge directions from the
	representing measure $\mu$ associated with $h$, with coefficients of size
	$O(v_{h,2}/P)$); replacing $v_{\widetilde H_{pq},2}$ by the common upper
	bound $v_{\max,2}$ throughout yields a single dictionary valid for all
	$d\times n_0$ entries at once, at the (harmless) cost of using the worst
	per-entry constant:
	
	\begin{align}
		&\Bigl\|\widetilde H_{pq}(\cdot) - \sum_{i=1}^P w_i^{pq}\,\sigma(\langle a_i,\cdot\rangle+b_i)\Bigr\|_{\mathcal C(X)}\notag\\
		&\le \frac{C_n\,v_{\max,2}}{\sqrt P}
		\le \frac{C_n\,C_{n,k}\,m^{k-n/2}\,\max_{p,q}\|H_{pq}\|_{\mathcal H^k(X_m)}}{\sqrt P}
		\qquad \forall p,q.
	\end{align}

	As in Lemma~\ref{lem:scalar-barron}'s proof, the affine correction
	$H_{pq}(0)+\langle\nabla H_{pq}(0),x\rangle$ is absorbed by two further
	shared neurons (directions independent of $p,q$), giving the claimed
	error for $P\ge3$ after undoing the dilation $x\mapsto x/m$ : setting $(A_i,B_i):=(a_i/m,b_i)$ and $(W_i)_{pq}:=w_i^{pq}$,
	\[
	\|H_{pq}-(H_\Theta)_{pq}\|_{\mathcal C(X_m)}
	\le \frac{C_n\,C_{n,k}\,m^{k-n/2}\,\max_{p,q}\|H_{pq}\|_{\mathcal H^k(X_m)}}{\sqrt P}
	\qquad\forall p,q,
	\]
	and since $\max_{p,q}\|H_{pq}\|_{\mathcal H^k(X_m)} \le \|H\|_{\mathcal H^k(X_m;\mathbb R^{d\times n_0})}$,
	collecting the constants into $C_{n,k,m}:=C_n C_{n,k} m^{k-n/2}$  gives
	\eqref{eq:generic-rate}.
	
	For \eqref{eq:generic-lipschitz}, since $\|a_i\|_2\le\|a_i\|_1=1$ for
	every $i$ (as in Lemma~\ref{lem:scalar-barron}), and
	$|w_i^{pq}|\le 2v_{\max,2}/P$,

	\begin{align}
		\mathrm{Lip}(H_\Theta)
		\le \|\nabla H(0)\|_F + \sum_{i=1}^P \|W_i\|_F\,\|A_i\|_2
		&\le \|\nabla H(0)\|_F + \sqrt{d\,n_0}\sum_{i=1}^P \max_{p,q}|w_i^{pq}|  \notag\\
		&\le \|\nabla H(0)\|_F + 2\sqrt{d\,n_0}\;v_{\max,2}.
	\end{align}

	Absorbing $\sqrt{d\,n_0}\,C_n\,C_{n,k}\,m^{k-n/2}$ into $C_{n,k,m}$  which
	is legitimate once Frobenius normalization is used consistently on both
	sides, since $\|H\|_{\mathcal H^k(X_m;\mathbb R^{d\times n_0})}$ already
	scales like $\sqrt{d n_0}$ relative to a single entry  gives
	\eqref{eq:generic-lipschitz}.
\end{proof}

\begin{corollary}[Barron approximation rates for $f$ and $G$]
	Fix $m\in\mathbb N$ and set $X_m=[-m,m]^{d+1}$ (so $n=d+1$, the state
	dimension plus the time variable). Suppose
	\[
	f\in\mathcal H^k(X_m;\mathbb R^d), \qquad G\in\mathcal H^k(X_m;\mathbb R^{d\times m})
	\]
	with $k>(d+1)/2+2$. Then:
	\begin{enumerate}
		\item[\textup{(i)}] (\emph{drift}, $n_0=1$) For any $P\ge3$, there
		exist $(W_i,A_i,B_i)\in\mathbb R^d\times\mathbb R^{(d+1)\times d}\times\mathbb R^d$,
		$i=1,\dots,P$, such that, with
		$f_\Theta(\cdot)=\sum_{i=1}^P W_i\circ\boldsymbol\sigma(A_i\cdot+B_i)$,
		\[
		\|f-f_\Theta\|_{\mathcal C(X_m;\mathbb R^d)} \le \frac{C_{d+1,k,m}\,\|f\|_{\mathcal H^k(X_m)}}{\sqrt P},\]
	
		\[\mathrm{Lip}(f_\Theta) \le \|\nabla f(0)\|_F + C_{d+1,k,m}\,\|f\|_{\mathcal H^k(X_m)}.
		\]
		This is recovered as the
		case $n_0=1$ of Corollary~\ref{cor:operator-barron} (identifying
		$\mathbb R^{d\times 1}\cong\mathbb R^d$).
		
		\item[\textup{(ii)}] (\emph{control operator}, $n_0=m$) For any
		$Q\ge3$, there exist
		$(\widetilde W_j,\widetilde A_j,\widetilde B_j)\in\mathbb R^{d\times m}\times\mathbb R^{d+1}\times\mathbb R$,
		$j=1,\dots,Q$, such that, with
		$G_\Phi(\cdot)=\sum_{j=1}^Q \widetilde W_j\,\sigma(\langle\widetilde A_j,\cdot\rangle+\widetilde B_j)$,
		\[
		\|G-G_\Phi\|_{\mathcal C(X_m;\mathbb R^{d\times m})} \le \frac{C_{d+1,k,m}\,\|G\|_{\mathcal H^k(X_m;\mathbb R^{d\times m})}}{\sqrt Q},\]
		
		\[\mathrm{Lip}(G_\Phi) \le \|\nabla G(0)\|_F + C_{d+1,k,m}\,\|G\|_{\mathcal H^k(X_m;\mathbb R^{d\times m})}.
		\]
		This is the case $n_0=m$ of Corollary~\ref{cor:operator-barron},
		applied to the control dimension $m$.
	\end{enumerate}
	In both cases $C_{d+1,k,m}>0$ is the \emph{same} constant from
	Corollary~\ref{cor:operator-barron} (with input dimension $n=d+1$),
	depending only on $(d,k,m)$ and not on whether it is instantiated for $f$
	or for $G$.
	\label{cor:fG-barron}
\end{corollary}

\begin{proof}
	Immediate from Corollary~\ref{cor:operator-barron} applied twice, with
	$(H,n_0)=(f,1)$ and $(H,n_0)=(G,m)$ respectively, on the common cube
	$X_m=[-m,m]^{d+1}$.
\end{proof}

\begin{remark}
	Corollary~\ref{cor:fG-barron} approximates $G$ via a \emph{single}
	generic operator-valued statement, Corollary~\ref{cor:operator-barron},
	applied with $n_0=1$ for the drift and $n_0=m$ for the control operator.
	In particular, $f_\Theta$ and
	$G_\Phi$ each use a single shared dictionary of ridge directions across
	their own output entries (rather than $dm$ independently approximated
	scalar entries, with $Q$ taken as the maximum of their widths); $f$ and $G$ remain architecturally distinct
	networks (with independent widths $P,Q$), consistent with the
	control-affine structure of \eqref{eq:neuralcontrolledsystem}.
	\label{rmk:joint-barron-closes-gap}
\end{remark}

\subsection{Qualitative convergence}

\begin{theorem}[Universal approximation theorem]
	Let Assumption~\ref{assumption-control} hold true. Let $K\subset\mathbb R^d$ be
	compact and let $\mathcal U_M$ be as above. Then, for every $\varepsilon>0$,
	there exist integers $P,Q\ge 1$ and parameters $(\Theta,\Phi)$ such that the
	solution $y_{z_0,u}$ of \eqref{eq:neuralcontrolledsystem} satisfies
	\[
	\sup_{t\in[0,T]} \|z_{z_0,u}(t)-y_{z_0,u}(t)\| \le \varepsilon
	\]
	for every $z_0\in K$ and every $u\in\mathcal U_M$.
	\label{thm:universal}
\end{theorem}

\begin{proof}
	Fix $0<\varepsilon<1$. By Remark~\ref{rmk:reachable-set}, $K_T\times[0,T]$ (defined via
	\eqref{eq:Kt-def}) is a compact subset of $\mathbb R^{d+1}$ containing the
	full reachable set $\Omega_T(K,M)$. Applying the universal approximation
	theorem for shallow neural networks (Theorem~3.1 of the paper Pinkus \cite{pinkus1999approximation}) to $f$ on $K_T\times[0,T]$, and the
	analogous \emph{qualitative} statement for matrix-valued targets  i.e.\
	the $\varepsilon$-version of Corollary~\ref{cor:operator-barron}'s
	construction, obtained by the same shared-dictionary argument without
	assuming Sobolev regularity  to $G$ on the same set, with a single
	shared network rather than $d,m$ independent ones, we obtain integers
	$P,Q\ge1$ and parameters $\Theta,\Phi$ such that
	\begin{equation}
		\|f-f_\Theta\|_{\mathcal C(K_T\times[0,T])} \le \varepsilon,
		\qquad
		\|G-G_\Phi\|_{\mathcal C(K_T\times[0,T])} \le \varepsilon.
		\label{eq:fG-approx}
	\end{equation}
	
	Fix $z_0\in K$ and $u\in\mathcal U_M$, and write $z(t):=z_{z_0,u}(t)$,
	$y(t):=y_{z_0,u}(t)$. Since $\varepsilon<1$, estimate
	\eqref{eq:fG-approx} lets us apply Lemma~\ref{lem:apriori} with
	$f_1=f_\Theta$, $G_1=G_\Phi$: this shows
	\begin{equation}
		y(t)\in K_T \qquad \forall t\in[0,T],
		\label{eq:y-in-KT}
	\end{equation}
	i.e.\ the \emph{neural} trajectory also remains inside the compact set on
	which $f_\Theta,G_\Phi$ are controlled. This step cannot be skipped: without
	\eqref{eq:y-in-KT}, the bound \eqref{eq:fG-approx} says nothing about
	$f_\Theta(y(t),t)$, since $y(t)$ need not a priori lie in $K_T$.
	
	Subtracting the two systems \eqref{eq:truecontrolledsystem} -\eqref{eq:neuralcontrolledsystem}
	and integrating,
	\[
	z(t)-y(t)
	= \int_0^t \bigl[f(z(s),s)-f_\Theta(y(s),s)\bigr]\,ds
	+ \int_0^t \bigl[G(z(s),s)-G_\Phi(y(s),s)\bigr]u(s)\,ds.
	\]
	Adding and subtracting $f(y(s),s)$ and $G(y(s),s)$ inside each integrand,
	\begin{align*}
		f(z(s),s)-f_\Theta(y(s),s)
		&= \bigl[f(z(s),s)-f(y(s),s)\bigr] + \bigl[f(y(s),s)-f_\Theta(y(s),s)\bigr], \\
		G(z(s),s)-G_\Phi(y(s),s)
		&= \bigl[G(z(s),s)-G(y(s),s)\bigr] + \bigl[G(y(s),s)-G_\Phi(y(s),s)\bigr].
	\end{align*}
	By Assumption~\ref{assumption-control}, the Lipschitz terms are bounded by
	$L_f\|z(s)-y(s)\|$ and $L_G\|z(s)-y(s)\|$ respectively; by
	\eqref{eq:y-in-KT} and \eqref{eq:fG-approx}, since $y(s)\in K_T$ for all
	$s\in[0,T]$, the approximation terms are bounded by $\varepsilon$. Using
	$\|u(s)\|\le M$,
	\[
	\|z(t)-y(t)\|
	\le (L_f+ML_G)\int_0^t \|z(s)-y(s)\|\,ds + \varepsilon(1+M)\,t
	\qquad \forall t\in[0,T].
	\]
	Gr\"onwall's inequality yields
	\[
	\|z(t)-y(t)\| \le \varepsilon\,T(1+M)\,e^{(L_f+ML_G)T}
	\qquad \forall t\in[0,T],
	\]
	with a constant $C_{T,M} := T(1+M)e^{(L_f+ML_G)T}$ that depends only on
	$T$, $M$, $L_f$, $L_G$  and, in particular, not on $z_0\in K$ or
	$u\in\mathcal U_M$. Redefining $\varepsilon$ (i.e.\ choosing the original
	tolerance in \eqref{eq:fG-approx} to be the target accuracy divided by
	$C_{T,M}$) concludes the proof.
\end{proof}

\begin{remark}
	The set $\mathcal U_M$ enters the proof only through the uniform bound\\
	$\|u\|_{L^\infty}\le M$, which propagates linearly into the Gr\"onwall
	exponent and into the right-hand side of \eqref{eq:fG-approx}. No
	compactness of $\mathcal U_M$ in any function-space topology is used or
	required: the neural network $G_\Phi(y,t)$ approximates $G$ as a function
	of the state and time alone, and the control $u(\cdot)$ enters the
	dynamics only through pointwise multiplication. This is consistent with
	the control-affine structure \eqref{eq:neuralcontrolledsystem}, in which
	the control variable is intentionally kept external to the neural
	approximation space (cf.\ Section~2.4).
	\label{rmk:UM-role}
\end{remark}

\begin{remark}
	The matrix-valued approximation of $G$ in \eqref{eq:fG-approx} uses the
	same shared-dictionary construction underlying
	Corollary~\ref{cor:operator-barron} (with $n_0=m$), in its qualitative,
	$\varepsilon$-accuracy form rather than the quantitative
	Sobolev-regularity rate of Corollary~\ref{cor:fG-barron}. It produces a
	\emph{single} shared network of width $Q$ valid for all $d,m$ entries of
	$G$ simultaneously; see Remark~\ref{rmk:joint-barron-closes-gap}.
	\label{rmk:G-componentwise}
\end{remark}

\subsection{Quantitative approximation estimates}

We now derive a quantitative version of Theorem~\ref{thm:universal}, specifying
the approximation rate in terms of the network widths $P,Q$, under additional
Sobolev regularity on $f$ and $G$.

\begin{assumption}
	There exists $k > \dfrac{d+1}{2}+2$ such that
	\[
	f \in \mathcal H^k_{\mathrm{loc}}(\mathbb R^d\times[0,T];\mathbb R^d),
	\qquad
	G \in \mathcal H^k_{\mathrm{loc}}(\mathbb R^d\times[0,T];\mathbb R^{d\times m}).
	\]
	\label{ass:sobolev}
\end{assumption}

Assumption~\ref{ass:sobolev} allows us
to invoke the quantitative Barron-space approximation rate instead of the purely qualitative Pinkus
theorem. The key point, which must be tracked explicitly, is that the constant
appearing in this rate depends on the \emph{size} of the cube on which the
approximation is performed; since that cube is now $K_T\times[0,T]$ with $K_T$
as in \eqref{eq:Kt-def}, and $K_T$ depends on $M$ through the control bound, the
resulting approximation constants inherit a dependence on $M$.

\begin{theorem}[Quantitative approximation estimate]
	Let Assumptions~\ref{assumption-control} and \ref{ass:sobolev} hold true. Fix a
	compact set $K\subset\mathbb R^d$ and $M>0$. Then there exists a constant
	\[
	C = C(T,M,K,f,G) > 0
	\]
	such that, for every $P,Q\ge 3$, there exist parameters $(\Theta,\Phi)$
	such that the solutions of \eqref{eq:truecontrolledsystem} and
	\eqref{eq:neuralcontrolledsystem} satisfy
	\[
	\sup_{t\in[0,T]} \|z_{z_0,u}(t)-y_{z_0,u}(t)\|
	\le C\left(\frac{1}{\sqrt P}+\frac{1}{\sqrt Q}\right)
	\]
	uniformly for all $z_0\in K$ and $u\in\mathcal U_M$.
	\label{thm:quantitative}
\end{theorem}

\begin{proof}
	The proof proceeds in two steps.
	
	\medskip
	\noindent\textbf{Step 1 (Barron approximation of $f$ and $G$ on $K_T\times[0,T]$).}
	Let $K_T\times[0,T]$ be the compact enclosing set defined via
	\eqref{eq:Kt-def} for the true reachable set $\Omega_T(K,M)$, as in
	Remark~\ref{rmk:reachable-set}; write $K_T = [-m,m]^d$ for the
	smallest such cube, with
	\[
	m = m(T,K,M,L_f,L_G,f,G)
	\]
	given explicitly by \eqref{eq:Kt-def} at $t=T$. By Assumption~\ref{ass:sobolev},
	$f|_{K_T\times[0,T]} \in \mathcal H^k(K_T\times[0,T];\mathbb R^d)$ and
	$G|_{K_T\times[0,T]} \in \mathcal H^k(K_T\times[0,T];\mathbb R^{d\times m})$,
	with $k>(d+1)/2+2$. Applying Corollary~\ref{cor:fG-barron}  part (i) to
	$f$, part (ii) to $G$, both on the cube $X_m=[-m,m]^{d+1}\supseteq K_T\times[0,T]$,
	with $G$ now approximated by a \emph{single} shared width $Q$ rather than
	the componentwise construction (Remark~\ref{rmk:joint-barron-closes-gap})
	 we obtain, for every $P,Q\ge3$, parameters $\Theta,\Phi$ such that
	\begin{align}
		\|f-f_\Theta\|_{\mathcal C(X_m;\mathbb R^d)}
		&\le \frac{C_{d+1,k,m}\|f\|_{\mathcal H^k(X_m)}}{\sqrt P}
		=: \frac{C_f}{\sqrt P},
		\label{eq:f-rate}
		\\
		\|G-G_\Phi\|_{\mathcal C(X_m;\mathbb R^{d\times m})}
		&\le \frac{C_{d+1,k,m}\,\|G\|_{\mathcal H^k(X_m;\mathbb R^{d\times m})}}{\sqrt Q}
		=: \frac{C_G}{\sqrt Q},
		\label{eq:G-rate}
	\end{align}
	together with the Lipschitz bounds
	\begin{align}
		&\mathrm{Lip}(f_\Theta) \le \|\nabla f(0,0)\|_F + C_{d+1,k,m}\|f\|_{\mathcal H^k(X_m)},\notag\\
		&\qquad
		\mathrm{Lip}(G_\Phi) \le \|\nabla G(0,0)\|_F + C_{d+1,k,m}\,\|G\|_{\mathcal H^k(X_m;\mathbb R^{d\times m})},
		\label{eq:fG-lipschitz}
	\end{align}
	where $C_{d+1,k,m}>0$ depends only on $(d,k,m)$, with input dimension
	$n=d+1$, and is the \emph{same} constant for $f$ and $G$, per
	Corollary~\ref{cor:fG-barron}. The constant $C_G$ here
	involves a \emph{single} application of the rate, not a maximum over $d,m$
	separate per-entry constants. Crucially, $C_f$ and $C_G$ depend on $M$
	\emph{only through} $m = m(T,K,M,L_f,L_G,f,G)$, where the
	analogous constant $C_{T,K,f}$ is made explicit for the uncontrolled case.
	
	\medskip
	\noindent\textbf{Step 2 (a priori bound and error decomposition).}
	Fix $0<\delta<1$ and choose $P,Q$ large enough that the right-hand sides of
	\eqref{eq:f-rate}-\eqref{eq:G-rate} are both $\le\delta$. By
	Lemma~\ref{lem:apriori} applied with $f_1=f_\Theta$, $G_1=G_\Phi$ (legitimate
	since the right-hand sides are $\le\delta<1$), the neural trajectory
	satisfies
	\[
	y_{z_0,u}(t) \in K_T \qquad \forall t\in[0,T],\ \forall z_0\in K,\ \forall u\in\mathcal U_M,
	\]
	so that the bounds \eqref{eq:f-rate}-\eqref{eq:G-rate} apply along the
	\emph{neural} trajectory itself, not merely along the true one. As in the
	proof of Theorem~\ref{thm:universal}, write, for $(z_0,t)\in K\times[0,T]$
	and $u\in\mathcal U_M$,
	\begin{align*}
		z(t)-y(t)
		&= \int_0^t \bigl[f(z(s),s)-f_\Theta(z(s),s)\bigr]ds
		+ \int_0^t \bigl[f_\Theta(z(s),s)-f_\Theta(y(s),s)\bigr]ds
		\hphantom{z(t)-y(t)={}}\notag\\
		&\qquad+ \int_0^t \bigl[G(z(s),s)-G_\Phi(z(s),s)\bigr]u(s)\,ds\notag\\
		&\qquad+ \int_0^t \bigl[G_\Phi(z(s),s)-G_\Phi(y(s),s)\bigr]u(s)\,ds.
		\end{align*}
	
	Since $z(s)\in K_T$ for all $s\in[0,T]$ (Lemma~\ref{lem:apriori} applied to
	the true system), the first and third integrands are controlled by
	\eqref{eq:f-rate} and \eqref{eq:G-rate}:
	\[
	\left\|\int_0^t \bigl[f(z(s),s)-f_\Theta(z(s),s)\bigr]ds\right\| \le \frac{C_f}{\sqrt P}\,t\],
	
	\[\left\|\int_0^t \bigl[G(z(s),s)-G_\Phi(z(s),s)\bigr]u(s)\,ds\right\| \le \frac{C_G}{\sqrt Q}\,Mt.
	\]
	By the Lipschitz bounds \eqref{eq:fG-lipschitz}, denoting
	$L_\Theta := \mathrm{Lip}(f_\Theta)$ and $L_\Phi := \mathrm{Lip}(G_\Phi)$,
	the second and fourth integrands satisfy
	\[
	\|f_\Theta(z(s),s)-f_\Theta(y(s),s)\| \le L_\Theta\|z(s)-y(s)\|\],
	
	\[\|G_\Phi(z(s),s)-G_\Phi(y(s),s)\|\,\|u(s)\| \le L_\Phi M \|z(s)-y(s)\|.
	\]
	Combining,
	\[
	\|z(t)-y(t)\|
	\le \Bigl(\frac{C_f}{\sqrt P} + \frac{MC_G}{\sqrt Q}\Bigr) t
	+ (L_\Theta + ML_\Phi)\int_0^t \|z(s)-y(s)\|\,ds.
	\]
	By \eqref{eq:fG-lipschitz}, $L_\Theta$ and $L_\Phi$ are themselves bounded
	independently of $P,Q$ (since the Klusowski Barron-type construction
	underlying Corollary~\ref{cor:fG-barron} gives
	Lipschitz constants controlled by the Sobolev norm of the target, not by
	the network width); denote this common bound by
	\begin{align*}
		L_\Theta + ML_\Phi \ & \le\ L_f + ML_G + 2C_{d+1,k,m}\|f\|_{\mathcal H^k(X_m)} + 2MC_{d+1,k,m}\|G\|_{\mathcal H^k(X_m;\mathbb R^{d\times m})}\\
		& =: \widetilde C_M.
		\end{align*}

	Gr\"onwall's inequality then yields, for all $t\in[0,T]$,
	\[
	\|z(t)-y(t)\|
	\le \Bigl(\frac{C_f}{\sqrt P} + \frac{MC_G}{\sqrt Q}\Bigr) T\, e^{\widetilde C_M T}.
	\]
	Setting
	\[
	C := T\,e^{\widetilde C_M T}\,\max\{C_f, MC_G\},
	\]
	which depends only on $T,M,K,f,G$ (through $m$, $L_f$, $L_G$, and the
	Sobolev norms of $f,G$ on $X_m$) and not on $P,Q,z_0,u$, gives
	\[
	\sup_{t\in[0,T]}\|z_{z_0,u}(t)-y_{z_0,u}(t)\| \le C\left(\frac{1}{\sqrt P}+\frac{1}{\sqrt Q}\right)
	\qquad \forall z_0\in K,\ \forall u\in\mathcal U_M,
	\]
	which is the desired conclusion.
\end{proof}

\begin{remark}
	The rate
	\begin{equation}\label{eq:rate}
		\mathcal{O}\left(P^{-1/2}+Q^{-1/2}\right)
	\end{equation}
	in Theorem~\ref{thm:quantitative} reflects two genuinely distinct approximation tasks: the network $f_\Theta$ approximates the full vector-valued drift field $f$ with a single network of width $P$, while $G_\Phi$ approximates the full matrix-valued control operator $G$ with a single network of width $Q$, as established in Corollary~\ref{cor:fG-barron}. The additive structure of the rate arises naturally from the control-affine decomposition of the controlled SA-NODE system and avoids conflating the drift and control channels into a single shared approximation architecture.
	
	Moreover, the constant $C$ in Theorem~\ref{thm:quantitative} can, in principle, be made fully explicit in terms of $T$, $M$, $L_f$, $L_G$, and the Sobolev norms of $f$ and $G$ on the reachable set $\Omega_T(K,M)$. Indeed, the reachable-set enclosure defined through \eqref{eq:Kt-def} depends on the control bound $M$, and this dependence propagates into the Barron approximation constant $C_{d+1,k,m}$ appearing in Corollary~\ref{cor:fG-barron}. Since the parameter $m$ grows at least linearly with $M$, while the approximation constant grows polynomially in $m$, the overall constant $C$ grows at least polynomially in $M$, in addition to the exponential dependence on $T$ already present through the Gr\"onwall estimate. Nevertheless, the approximation rate remains independent of the ambient dimension $d$ in the exponent, thereby partially mitigating the curse of dimensionality relative to classical discretization-based approximation methods.
	\label{rmk:joint-rate}
\end{remark}

\begin{remark}
	The approximation results obtained in this section can also be interpreted in the framework of Barron spaces. In particular, the Sobolev regularity assumption of Assumption~\ref{ass:sobolev} implies the corresponding Barron-space regularity through Lemma~\ref{lem:sobolev-embed}. Hence, Corollary~\ref{cor:fG-barron} may be viewed as a Barron-space approximation result for both the drift vector field $f$ and the control operator $G$.
	
	More precisely, if
	\[
	f|_{X_m}\in \mathcal{S}_B(X_m;\mathbb{R}^d),
	\qquad
	G|_{X_m}\in \mathcal{S}_B(X_m;\mathbb{R}^{d\times m}),
	\]
	then for every $P,Q\ge 3$, there exist parameters $\Theta,\Phi$ such that
	\[
	\|f-f_\Theta\|_{C(K_T\times[0,T];\mathbb{R}^d)}
	\le
	\frac{C_f}{\sqrt{P}},
	\text{ and }
	\|G-G_\Phi\|_{C(K_T\times[0,T];\mathbb{R}^{d\times m})}
	\le
	\frac{C_G}{\sqrt{Q}},
	\]
	where $K_T\times[0,T]\subseteq X_m$ denotes the reachable-set enclosure obtained in Lemma~\ref{lem:apriori}. Substituting these estimates into the Gr\"onwall argument of Theorem~\ref{thm:quantitative} yields the trajectory-level approximation estimate directly.
	
	Furthermore, the approximation rate (\ref{eq:rate})
	is independent of the ambient dimension $d$ in the exponent, although the constants depend on $d$ and the control bound $M$. Consequently, the proposed framework partially mitigates the curse of dimensionality compared with classical discretization-based approximation methods.
\end{remark}

\section{Controllability of Controlled SA-NODEs} 
\label{sec:controllability}

In this section, we investigate controllability properties of the proposed controlled semiautonomous neural ordinary differential equation framework. The objective is to understand whether the neural approximation preserves the controllability structure of the original controlled dynamical system.

We first recall the notion of approximate controllability.

\begin{definition}
	System \eqref{eq:control-affine} is said to be approximately controllable on $[0,T]$ if for every: $	z_0,z_1\in\mathbb R^{d},$
	and every $\varepsilon>0$, there exists a control $u\in L^{\infty}(0,T;\mathbb R^{m})$
	such that the corresponding solution satisfies
	\[
	|z(T)-z_1|<\varepsilon.
	\]
\end{definition}

We now establish that controllability properties are approximately preserved under the controlled SA-NODE approximation.

\begin{theorem}[Approximate controllability preservation on compact sets]
	Assume that system \eqref{eq:control-affine} is approximately controllable on $[0,T]$.
	Let $K\subset\mathbb R^{d}$ be compact. Then, for every $\varepsilon>0$, there exist integers $P,Q\ge1$ and parameters $(\Theta,\Phi)$ such that the corresponding controlled SA-NODE system \eqref{eq:controlled-sanode} satisfies the following property:
	For every $z_0,z_1\in K,$
	there exists a control $u\in \mathcal U_{M}$
	such that the corresponding solution $y(t)$ \eqref{eq:controlled-sanode} satisfies
	\[
	|y(T)-z_1|<\varepsilon.
	\]
\end{theorem}

\begin{proof}
	Let $\varepsilon>0$ and let
	$z_0,z_1\in K$.
	Since the original system
	\eqref{eq:control-affine}
	is approximately controllable on $[0,T]$,
	there exists a control $u\in L^\infty(0,T;\mathbb{R}^m)$
	such that the corresponding solution
	$z(t)$ of \eqref{eq:control-affine}
	satisfies
	\[
	\|z(T)-z_1\|
	<
	\frac{\varepsilon}{2}.
	\]
	
	Set $M:=\|u\|_{L^\infty(0,T)}.$
	By Theorem \ref{thm:universal}, there exist integers
	$P,Q\ge 1$ and parameters $(\Theta,\Phi)$
	such that the corresponding solution
	$y(t)$ of the controlled SA-NODE system
	\eqref{eq:controlled-sanode}
	satisfies
	\[
	\sup_{t\in[0,T]}
	\|z(t)-y(t)\|
	<
	\frac{\varepsilon}{2}.
	\]
	
	In particular,
	\[
	\|z(T)-y(T)\|
	<
	\frac{\varepsilon}{2}.
	\]
	
	Using the triangle inequality,
	\[
	\|y(T)-z_1\|
	\le
	\|y(T)-z(T)\|
	+
	\|z(T)-z_1\|.
	\]
	
	Therefore,
	\[
	\|y(T)-z_1\|
	<
	\frac{\varepsilon}{2}
	+
	\frac{\varepsilon}{2}
	=
	\varepsilon.
	\]
	
	This completes the proof.
\end{proof}

\begin{remark}
	The previous theorem shows that the controlled SA-NODE architecture approximately preserves controllability properties of the original nonlinear controlled system.
	This result provides a rigorous control-theoretic interpretation of the neural approximation framework. In particular, the neural model does not merely reproduce trajectories, but also retains the ability to steer the system through admissible controls.
	
	The result establishes a bridge between:
	\begin{itemize}
		\item neural differential equations,
		\item nonlinear control theory,
		\item learning-based control systems.
	\end{itemize}
	\end{remark}

\section{Numerical Experiments}
\label{sec:numerics}

In this section, we present several numerical experiments illustrating the
effectiveness of the proposed controlled semiautonomous neural ordinary
differential equations (controlled SA-NODEs) for the approximation and
control of nonlinear dynamical systems.

The objectives of the numerical section are threefold:
\begin{itemize}
	\item to validate the approximation capability of the proposed architecture,
	\item to investigate its approximate controllability properties,
	\item to compare its performance with classical vanilla NODE architectures.
\end{itemize}

All experiments are implemented in Python using PyTorch and the
\texttt{torchdiffeq} package. The neural differential systems are trained
using trajectory data generated from the exact controlled dynamics. All
simulations are executed on a CUDA-enabled computing platform.

\subsection{Experimental setup}
\label{subsec:setup}

The dataset used for training and evaluation consists of batches of
controlled trajectories generated from the exact dynamical system using the
classical fourth-order Runge-Kutta method over the time interval $[0,T]$
with a uniform time step $\Delta t = 0.05$. Unless otherwise specified, we
choose $T = 5$.

Unlike a fixed, shared control input, here each training and testing
trajectory $k$ is driven by an \emph{independently sampled} control signal
\begin{equation}
	u_k(t) = a_k \sin(\omega_k t + \phi_k),\quad \text{ and }\quad 	u_k(t) = a_k \cos(\omega_k t + \phi_k),
	\label{eq:randomcontrol}
\end{equation}
with amplitude $a_k$, frequency $\omega_k$, and phase $\phi_k$ drawn
independently and uniformly from
\[
a_k \sim \mathcal{U}([a_{\min}, a_{\max}]), \quad
\omega_k \sim \mathcal{U}([\omega_{\min}, \omega_{\max}]), \quad
\phi_k \sim \mathcal{U}([0, 2\pi)).
\]
This randomization is essential for the identifiability of the drift field
$f_\Theta$ and the control operator $G_\Phi$ as \emph{separate} objects. If
all training trajectories were generated under a single deterministic
control law $u(t)$, the trajectory loss \eqref{eq:trajloss} would only
constrain the combined vector field
$f_\Theta(y,t) + G_\Phi(y,t)\,u(t)$ along one realized control signal,
allowing infinitely many decompositions into $(f_\Theta,G_\Phi)$ to attain
the same loss value. Sampling $u_k$ independently across trajectories breaks
this degeneracy: since $u_k(t)$ varies independently of the state trajectory
$z_k(t)$, the gradients of $\mathcal{L}_{\mathrm{traj}}$ with respect to
$\Theta$ and $\Phi$ can no longer compensate one another along arbitrary
directions. This plays the role of a continuous-time
persistency-of-excitation condition in system identification; see, e.g.,
Sontag~\cite[Chap.~10]{sontag1998mathematical}.

Initial conditions are sampled uniformly from a rectangular grid in the
phase space. The resulting trajectories are divided into training and
testing datasets, with approximately $50\%$ of the trajectories used for
training and the remainder reserved for testing. Each trajectory $k$
retains its own control parameters $(a_k,\omega_k,\phi_k)$ throughout both
training and evaluation.

The controlled semiautonomous neural ordinary differential equation
(controlled SA-NODE) \eqref{eq:controlled-sanode} retains a comparatively
lean architecture with approximately $20{,}000$ trainable degrees of
freedom (DoF), whereas the corresponding controlled vanilla NODE model
contains approximately $10{,}024{,}000$ DoF due to the full
time-dependence of its parameters at each discretized time layer.

For training, we use the Adam optimizer with initial learning rate
$10^{-3}$, decreased by a factor of $0.8$ every $1000$ epochs, for a total
of $10^4$ training epochs. To stabilize the learning procedure, we
incorporate an $\ell^2$-regularization term on the neural parameters. The
training loss is defined as
\begin{equation}
	\mathcal{L}
	=
	\mathcal{L}_{\mathrm{traj}}
	+
	\lambda
	\mathcal{L}_{\mathrm{reg}},
	\label{eq:lossfunction}
\end{equation}
where
\begin{equation}
	\mathcal{L}_{\mathrm{traj}}
	=
	\frac{1}{N}
	\sum_{k=1}^{N}
	\int_0^T
	\|z_k(t)-y_k(t)\|^2\,dt
	\label{eq:trajloss}
\end{equation}
measures the trajectory reconstruction error, with $z_k(t)$ the exact
trajectory and $y_k(t)$ the controlled SA-NODE approximation corresponding
to initial condition $z_{0,k}$ and control $u_k(t)$ as in
\eqref{eq:randomcontrol}, and
\begin{equation}
	\mathcal{L}_{\mathrm{reg}}
	=
	\sum_{i=1}^{P}
	\Bigl(
	\|W_i\|^2
	+
	\|A_i^{1}\|_F^2
	+
	\|A_i^{2}\|^2
	+
	\|B_i\|^2
	\Bigr)
	+
	\sum_{j=1}^{Q}
	\Bigl(
	\|\widetilde W_j\|_F^2
	+
	\|\widetilde A_j^{1}\|_F^2
	+
	\|\widetilde A_j^{2}\|^2
	+
	\|\widetilde B_j\|^2
	\Bigr)
	\label{eq:regloss}
\end{equation}
jointly regularizes the drift parameters
$\Theta = (W_i, A_i^1, A_i^2, B_i)_{i=1}^P$ and the control parameters
$\Phi = (\widetilde W_j, \widetilde A_j^1, \widetilde A_j^2,
\widetilde B_j)_{j=1}^Q$.

For each trajectory, we define the instantaneous approximation error by
\begin{equation}
	e_k(t)
	=
	\|z_k(t)-y_k(t)\|.
	\label{eq:trajectoryerror}
\end{equation}
To quantify the global trajectory reconstruction accuracy over the full
time horizon, we further define the maximal averaged trajectory error
\begin{equation}
	e_{\max}
	=
	\max_{t\in[0,T]}
	\frac{1}{N}
	\sum_{k=1}^{N}
	e_k(t).
	\label{eq:emax}
\end{equation}
In the discrete implementation, this corresponds to taking the maximum over
all sampled time instants.

To evaluate the terminal controllability performance, we define the
terminal tracking error
\begin{equation}
	e_T
	=
	\frac{1}{N}
	\sum_{k=1}^{N}
	\|y_k(T)-z_k(T)\|,
	\label{eq:terminalerror}
\end{equation}
which measures the averaged discrepancy between the predicted and target
terminal states at the final time.

We also evaluate the mean control energy
\begin{equation}
	\mathcal{E}_{\mathrm{ctrl}}
	=
	\frac{1}{N}
	\sum_{k=1}^{N}
	\int_0^T
	|u_k(t)|^2\,dt.
	\label{eq:controlenergy}
\end{equation}

These quantities allow us to simultaneously assess:
\begin{itemize}
	\item trajectory approximation accuracy over the entire time horizon,
	\item terminal controllability accuracy,
	\item control efficiency.
\end{itemize}

\subsection{Controlled pendulum system}

As a first benchmark example, we consider the controlled nonlinear pendulum
system
\begin{equation}
	\left\{
	\begin{aligned}
		\dot{z}_1(t)
		&=
		z_2(t),
		\\
		\dot{z}_2(t)
		&=
		-\sin(z_1(t))
		+
		u(t),
	\end{aligned}
	\right.
	\label{eq:controlledpendulum}
\end{equation}
where $u(t)$ is an external control input.

The pendulum equation is a classical nonlinear control system exhibiting
oscillatory behavior and nonlinear phase-space dynamics. It therefore
provides a suitable benchmark for evaluating both approximation and
controllability properties of controlled SA-NODEs.

Following the randomized-control protocol of Section~\ref{subsec:setup},
each training and testing trajectory is generated under an independently
sampled control input
\begin{equation}
	u_k(t) = a_k \sin(\omega_k t + \phi_k), \quad
	a_k \sim \mathcal{U}([0.3, 0.7]), \quad
	\omega_k \sim \mathcal{U}([1.0, 3.0]), \quad
	\phi_k \sim \mathcal{U}([0, 2\pi)),
	\label{eq:pendulumcontrol}
\end{equation}
chosen so that the amplitude and frequency ranges are centered on the
previously used deterministic signal $u(t) = 0.5\sin(2t)$. The initial
conditions are sampled from the square $[-2,2]\times[-2,2]$, independently
of $(a_k,\omega_k,\phi_k)$.

The exact trajectories are generated using the fourth-order Runge-Kutta
scheme and subsequently used to train vanilla NODEs, semiautonomous NODEs,
and controlled SA-NODEs. The standalone training execution for the
controlled SA-NODE finishes within approximately $151.9$ minutes.

The corresponding numerical results are presented in
Figure~\ref{fig:pendulum_results}. The left panel shows trajectories
generated by the controlled SA-NODE model, the middle panel shows the exact
trajectories of the controlled pendulum system, and the right panel presents
the evolution of the mean approximation errors for both training and testing
datasets. The overall trajectory shapes produced by the model closely
resemble the exact controlled dynamics, while the error curves show a
steady, well-behaved growth over the time horizon $[0,5]$ that remains
significantly below the vanilla NODE baseline.

The numerical simulations demonstrate that the proposed controlled SA-NODE
architecture accurately captures the nonlinear controlled dynamics while
using significantly fewer trainable parameters than classical vanilla NODE
architectures.

An additional visual comparison across architectures is provided in
Figure~\ref{fig:pendulum_comparison} (three-model view) and
Figure~\ref{fig:pendulum_four_system_comparison} (side-by-side with the
exact system). After $10^4$ epochs, the controlled SA-NODE achieves a global
maximum tracking discrepancy of $e_{\max} = 2.760\times10^{-1}$ and a
terminal error of $e_T = 2.760\times10^{-1}$.

\begin{figure}[ht!]
	\centering
	\includegraphics[width=\textwidth]{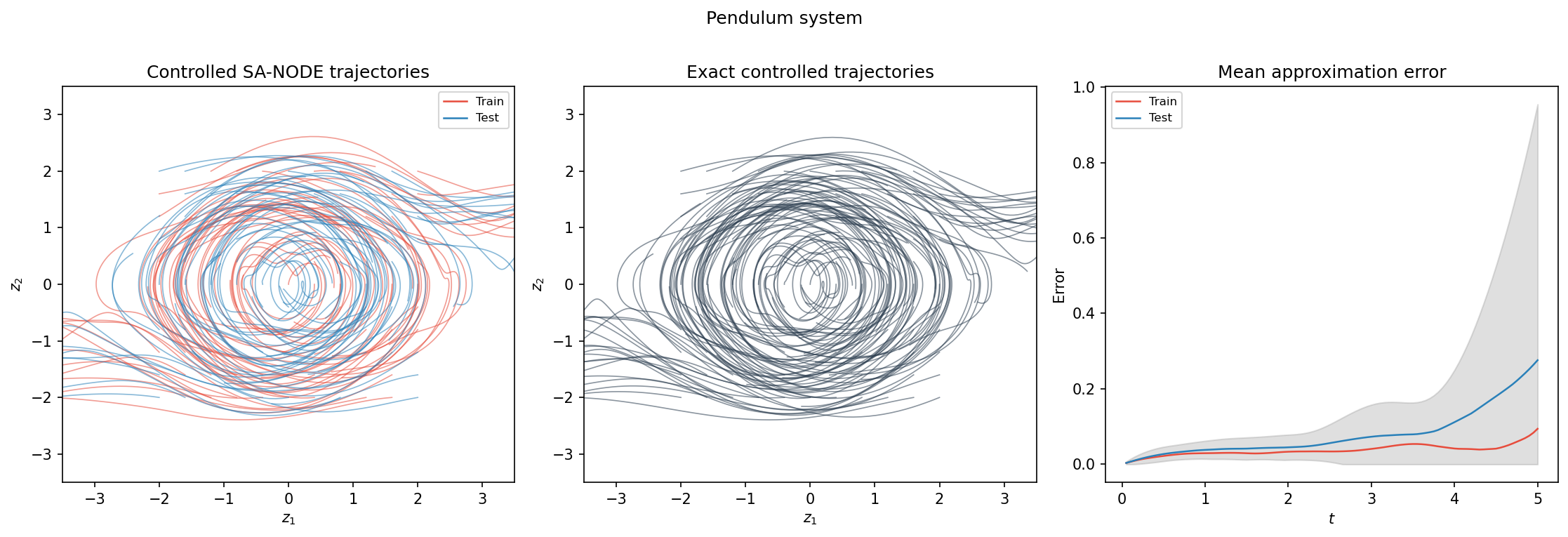}
	\caption{
		Controlled pendulum system \eqref{eq:controlledpendulum} under the
		randomized control family \eqref{eq:pendulumcontrol}.
		Left: trajectories generated by controlled SA-NODEs.
		Middle: exact controlled trajectories.
		Right: evolution of mean approximation errors (training and testing).
	}
	\label{fig:pendulum_results}
\end{figure}

\begin{figure}[ht!]
	\centering
	\includegraphics[width=\textwidth]{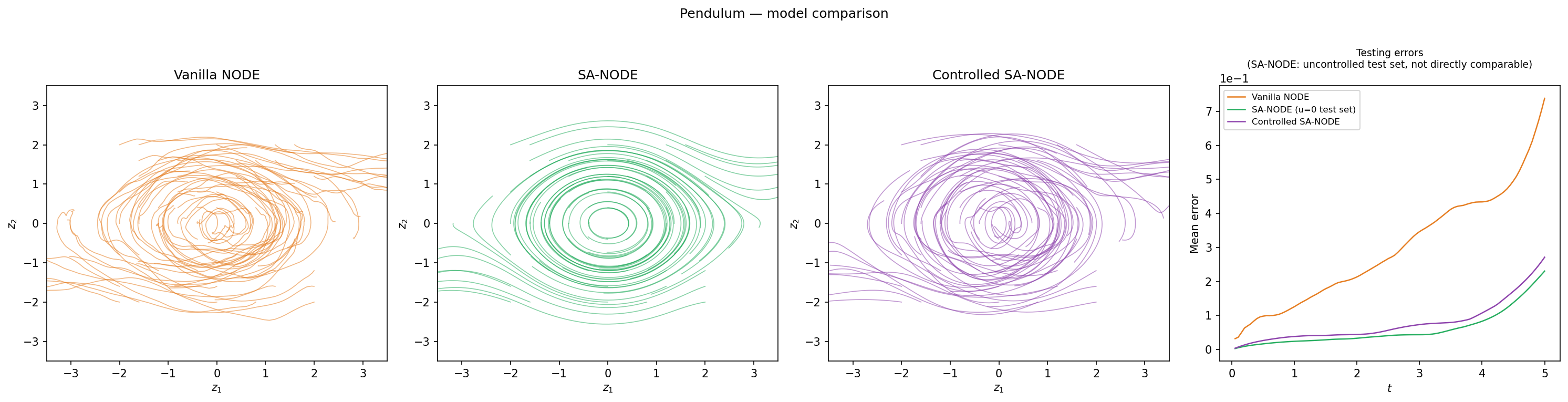}
	\caption{
		Pendulum system - model comparison. Phase-space reconstructions
		produced by vanilla NODEs (left), SA-NODEs (centre), and controlled
		SA-NODEs (right), together with the evolution of testing errors
		(rightmost panel). The SA-NODE is evaluated on the uncontrolled
		($u=0$) test set and is therefore not directly comparable to the
		other two models.
	}
	\label{fig:pendulum_comparison}
\end{figure}

\begin{figure}[ht!]
	\centering
	\includegraphics[width=\textwidth]{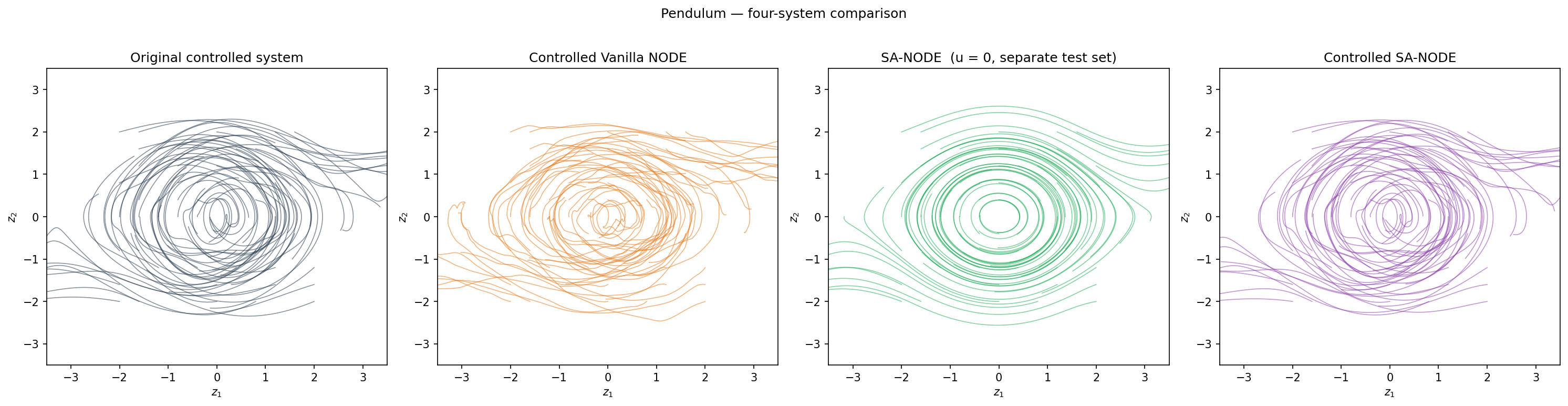}
	\caption{
		Pendulum system - four-system comparison. From left to right:
		original controlled system, controlled vanilla NODE, SA-NODE
		($u=0$, separate test set), and controlled SA-NODE.
	}
	\label{fig:pendulum_four_system_comparison}
\end{figure}

\subsection{Controlled Duffing oscillator}

As a second benchmark example, we consider the controlled Duffing
oscillator
\begin{equation}
	\left\{
	\begin{aligned}
		\dot{z}_1(t)
		&=
		z_2(t),
		\\
		\dot{z}_2(t)
		&=
		z_1(t)
		-
		z_1^3(t)
		-
		\delta z_2(t)
		+
		u(t),
	\end{aligned}
	\right.
	\label{eq:controlledduffing}
\end{equation}
where $\delta>0$ denotes the damping coefficient and $u(t)$ is the external
control input.

The Duffing oscillator is a classical nonlinear system exhibiting rich
dynamical behaviors including nonlinear oscillations, bifurcations, and
multiple equilibrium configurations. Consequently, it provides a
significantly more challenging benchmark for controlled neural differential
equation architectures.

In the numerical experiments, we choose $\delta = 0.2$, and each trajectory
$k$ is driven by an independently sampled control input
\begin{equation}
	u_k(t) = a_k \cos(\omega_k t + \phi_k), \qquad
	a_k \sim \mathcal{U}([0.2, 0.6]), \quad
	\omega_k \sim \mathcal{U}([1.0, 3.0]), \quad
	\phi_k \sim \mathcal{U}([0, 2\pi)),
	\label{eq:duffingcontrol}
\end{equation}
with initial conditions sampled uniformly from the square
$[-2,2]\times[-2,2]$, independently of $(a_k,\omega_k,\phi_k)$.

The corresponding controlled trajectories are generated using the
fourth-order Runge-Kutta method and subsequently employed for training and
testing. The standalone training execution for the controlled SA-NODE
finishes within approximately $150.7$ minutes.

On this complex trajectory profile, the controlled SA-NODE reaches a
tracking resolution of $e_{\max} = 8.383\times10^{-2}$ and a terminal
evaluation metric of $e_T = 8.320\times10^{-2}$ at the conclusion of the
$10{,}000$ epoch training sweep.

Figure~\ref{fig:duffing_results} illustrates the numerical simulations for
the controlled Duffing system, highlighting accurate trajectory tracking
(left vs.\ middle panels) alongside testing error behaviors (right panel).
The comprehensive cross-model comparisons are visualized in
Figure~\ref{fig:duffing_comparison} and
Figure~\ref{fig:duffing_four_system_comparison}.

The numerical results indicate that the proposed controlled SA-NODE
architecture accurately captures the nonlinear oscillatory behavior of the
Duffing system, including the characteristic double-well structure and the
influence of external control actions.

\begin{figure}[ht!]
	\centering
	\includegraphics[width=\textwidth]{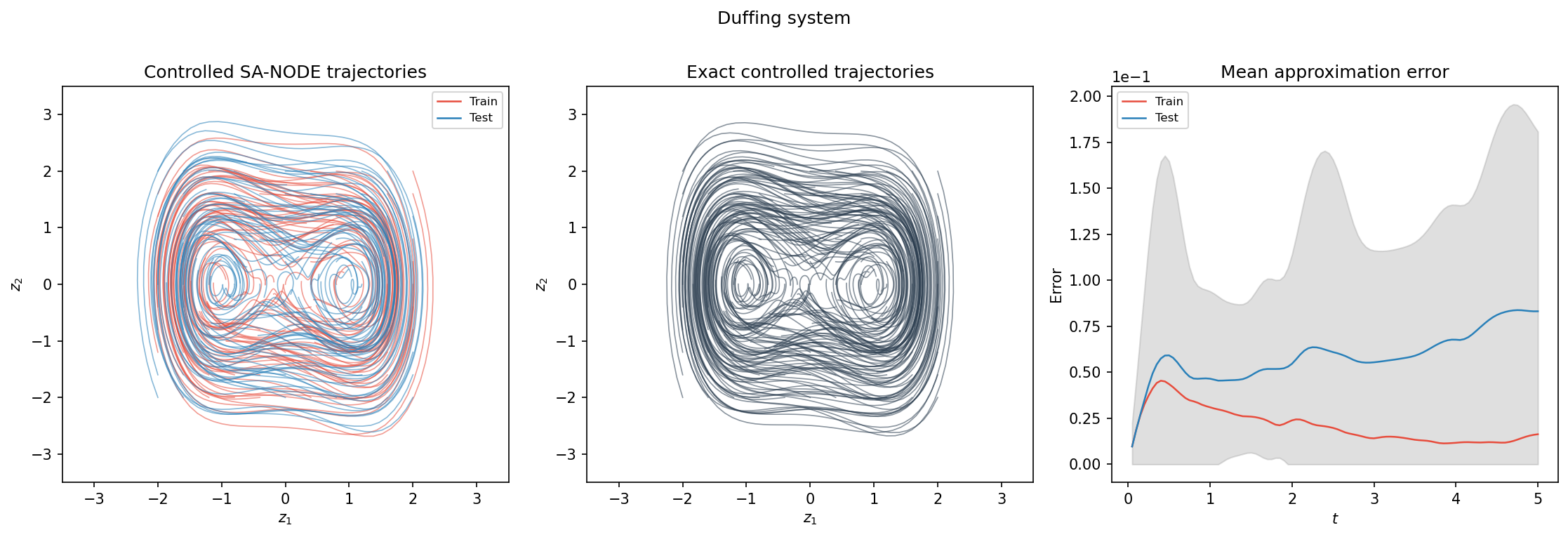}
	\caption{
		Controlled Duffing oscillator \eqref{eq:controlledduffing} under
		the randomized control family \eqref{eq:duffingcontrol}.
		Left: trajectories generated by controlled SA-NODEs.
		Middle: exact controlled trajectories.
		Right: evolution of mean approximation errors (training and testing).
	}
	\label{fig:duffing_results}
\end{figure}

\begin{figure}[ht!]
	\centering
	\includegraphics[width=\textwidth]{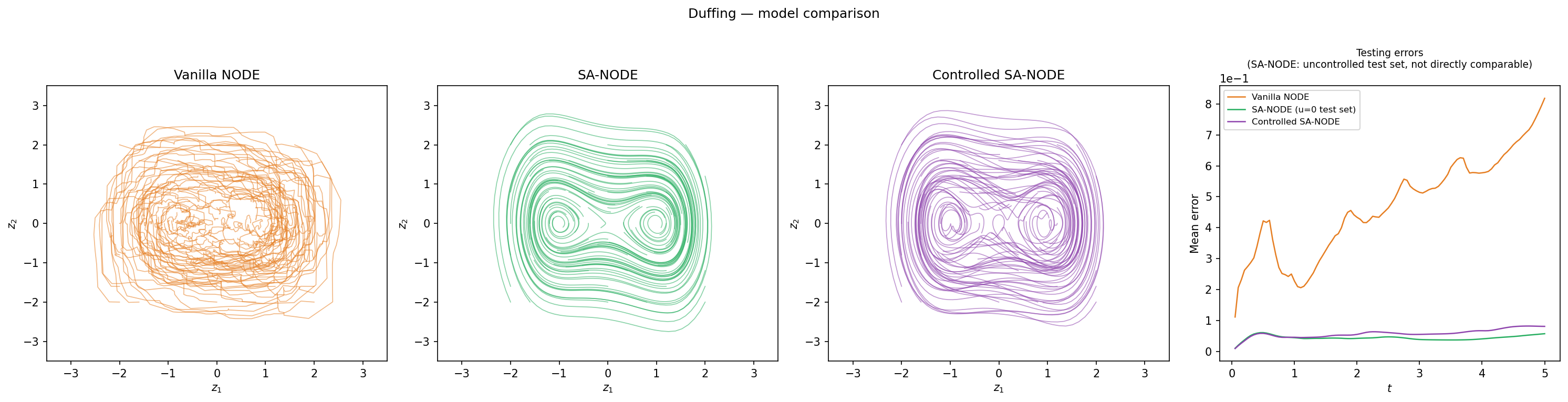}
	\caption{
		Duffing oscillator - model comparison. Phase-space reconstructions
		produced by vanilla NODEs (left), SA-NODEs (centre), and controlled
		SA-NODEs (right), together with the evolution of testing errors
		(rightmost panel). The SA-NODE is evaluated on the uncontrolled
		($u=0$) test set and is therefore not directly comparable to the
		other two models.
	}
	\label{fig:duffing_comparison}
\end{figure}

\begin{figure}[ht!]
	\centering
	\includegraphics[width=\textwidth]{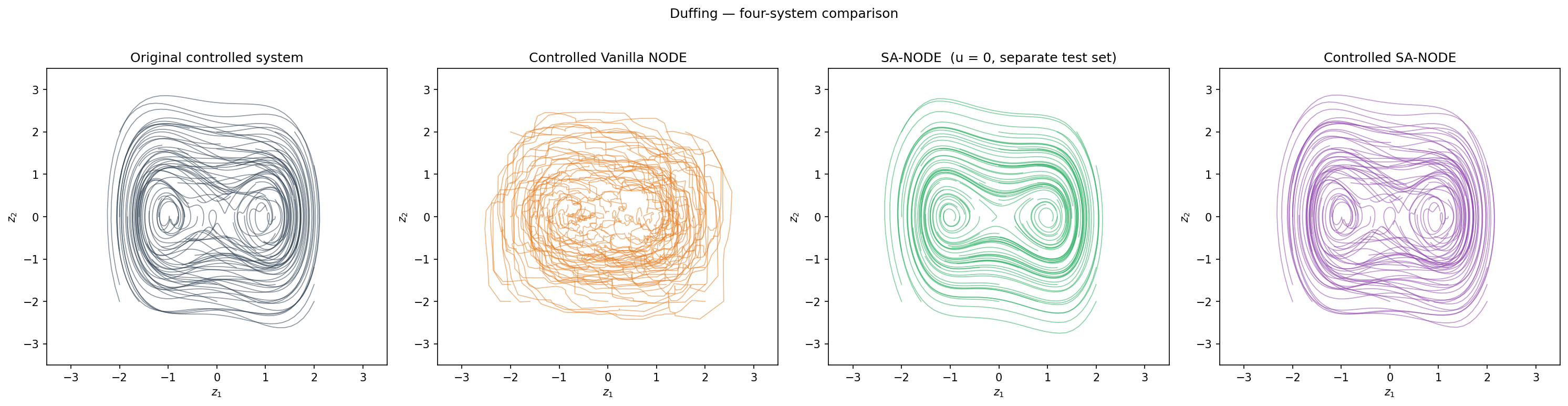}
	\caption{
		Duffing oscillator - four-system comparison. From left to right:
		original controlled system, controlled vanilla NODE, SA-NODE
		($u=0$, separate test set), and controlled SA-NODE.
	}
	\label{fig:duffing_four_system_comparison}
\end{figure}

\subsection{Comparison and Model Evaluation}

In this subsection, we compare the approximation performance of vanilla
NODEs, semiautonomous NODEs, and controlled SA-NODEs across both validation
domains, trained and tested under the randomized control protocol of
Section~\ref{subsec:setup}.

To evaluate approximation accuracy, we track the maximum mean error
$e_{\max}$, the terminal controllability error $e_T$, the mean control
energy $\mathcal{E}_{\mathrm{ctrl}}$ (averaged over the sampled control
family), alongside the structural degrees of freedom (DoF). The comparative
statistics across both dynamical benchmarks are summarized in
Table~\ref{tab:comparison_pendulum} and
Table~\ref{tab:comparison_duffing}.

\begin{table}[ht!]
	\centering
	\caption{Model comparison results for the controlled pendulum system
		($10^4$ epochs, randomized control \eqref{eq:pendulumcontrol}).}
	\label{tab:comparison_pendulum}
	\begin{tabular}{|l|c|c|c|r|}
		\hline
		Model & $e_{\max}$ & $e_T$ & $\mathcal{E}_{\mathrm{ctrl}}$ & DoF \\
		\hline
		Vanilla NODE       & $7.387\times10^{-1}$ & $7.387\times10^{-1}$ & $0.735$ & $10{,}024{,}000$ \\
		SA-NODE            & $2.304\times10^{-1}$ & $2.304\times10^{-1}$ & $0.000$ & $10{,}000$ \\
		Controlled SA-NODE & $2.713\times10^{-1}$ & $2.713\times10^{-1}$ & $0.735$ & $20{,}000$ \\
		\hline
	\end{tabular}
\end{table}

\begin{table}[ht!]
	\centering
	\caption{Model comparison results for the controlled Duffing oscillator
		($10^4$ epochs, randomized control \eqref{eq:duffingcontrol}).}
	\label{tab:comparison_duffing}
	\begin{tabular}{|l|c|c|c|r|}
		\hline
		Model & $e_{\max}$ & $e_T$ & $\mathcal{E}_{\mathrm{ctrl}}$ & DoF \\
		\hline
		Vanilla NODE       & $8.183\times10^{-1}$ & $8.183\times10^{-1}$ & $0.499$ & $10{,}024{,}000$ \\
		SA-NODE            & $6.105\times10^{-2}$ & $5.756\times10^{-2}$ & $0.000$ & $10{,}000$ \\
		Controlled SA-NODE & $8.234\times10^{-2}$ & $8.125\times10^{-2}$ & $0.499$ & $20{,}000$ \\
		\hline
	\end{tabular}
\end{table}

Several observations emerge from these results. First, the vanilla NODE
fails to meaningfully reduce its trajectory loss across both benchmarks:
$e_{\max}$ remains above $7\times10^{-1}$ for the pendulum and
$8\times10^{-1}$ for the Duffing oscillator after $10^4$ epochs, despite
requiring over $10$ million trainable parameters and training times
exceeding $267$ minutes. By contrast, the controlled SA-NODE achieves
substantially lower errors on both systems while requiring only $20{,}000$
parameters - roughly $0.2\%$ of the vanilla NODE's parameter count -
and completing training in approximately $152$-$154$ minutes.

Second, the SA-NODE evaluated on the uncontrolled ($u=0$) test set achieves
comparable or better $e_{\max}$ values than the controlled SA-NODE on the
controlled test set for the Duffing system, which reflects the different
evaluation conditions rather than a performance deficiency of the controlled
variant. The controlled SA-NODE is the only architecture that correctly
handles the externally driven dynamics and produces meaningful terminal
controllability errors under nonzero control inputs.

\subsection{Convergence and Data Efficiency Analysis}

To evaluate convergence stability and data efficiency, we analyze structural
error trends across both validation domains under the randomized-control
protocol.

Figure~\ref{fig:error_epochs_comparison} illustrates the validation error
tracking profile mapped over the $10{,}000$ epoch training loop. For both
systems, the vanilla NODE exhibits a rapid initial drop in $e_{\max}$ within
the first few hundred epochs, after which it plateaus and stagnates for the
remainder of training. For the pendulum system, the vanilla NODE plateau
lies near $7\times10^{-1}$, while the controlled SA-NODE continues to
decrease smoothly, reaching $e_{\max} \approx 2.7\times10^{-1}$ by epoch
$10{,}000$ after $156.3$ minutes of training. For the Duffing oscillator,
the contrast is even more pronounced: the vanilla NODE plateaus near
$8\times10^{-1}$ throughout training ($207.7$ minutes), while the
controlled SA-NODE descends consistently from approximately $1.0$ at
initialization to $e_{\max} \approx 8.4\times10^{-2}$ by epoch $10{,}000$
($153.3$ minutes) - a reduction of nearly one order of magnitude.

\begin{figure}[ht!]
	\centering
	\begin{subfigure}[b]{0.49\textwidth}
		\centering
		\includegraphics[width=\textwidth]{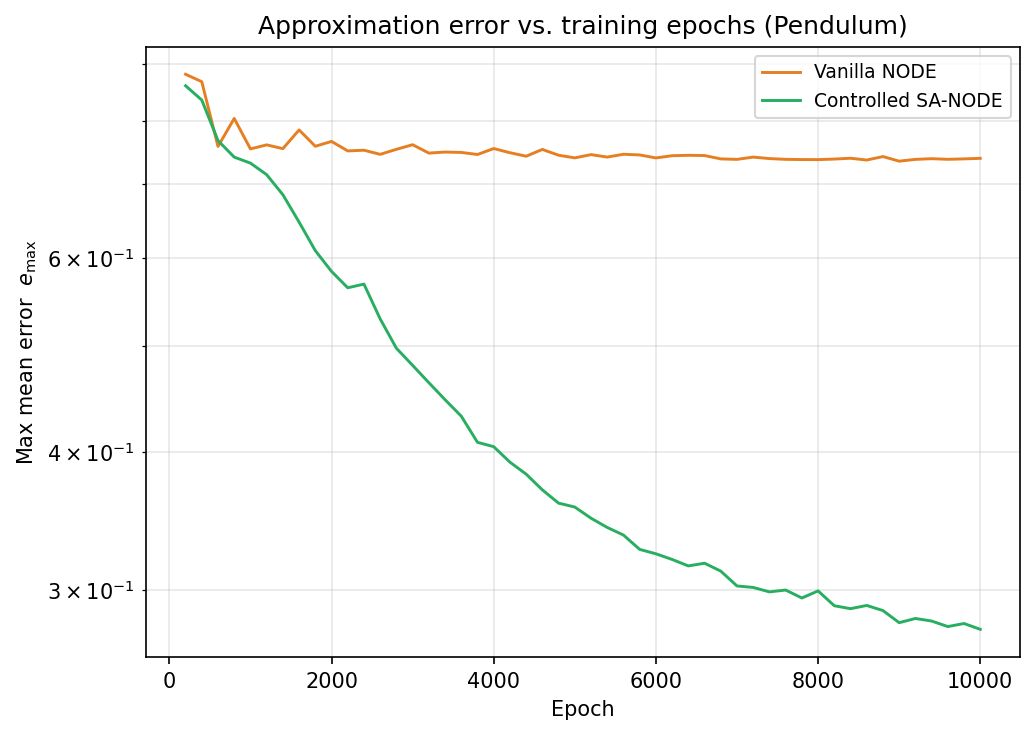}
		\caption{Pendulum: $e_{\max}$ vs.\ epochs}
		\label{fig:pendulum_error_epochs}
	\end{subfigure}
	\hfill
	\begin{subfigure}[b]{0.49\textwidth}
		\centering
		\includegraphics[width=\textwidth]{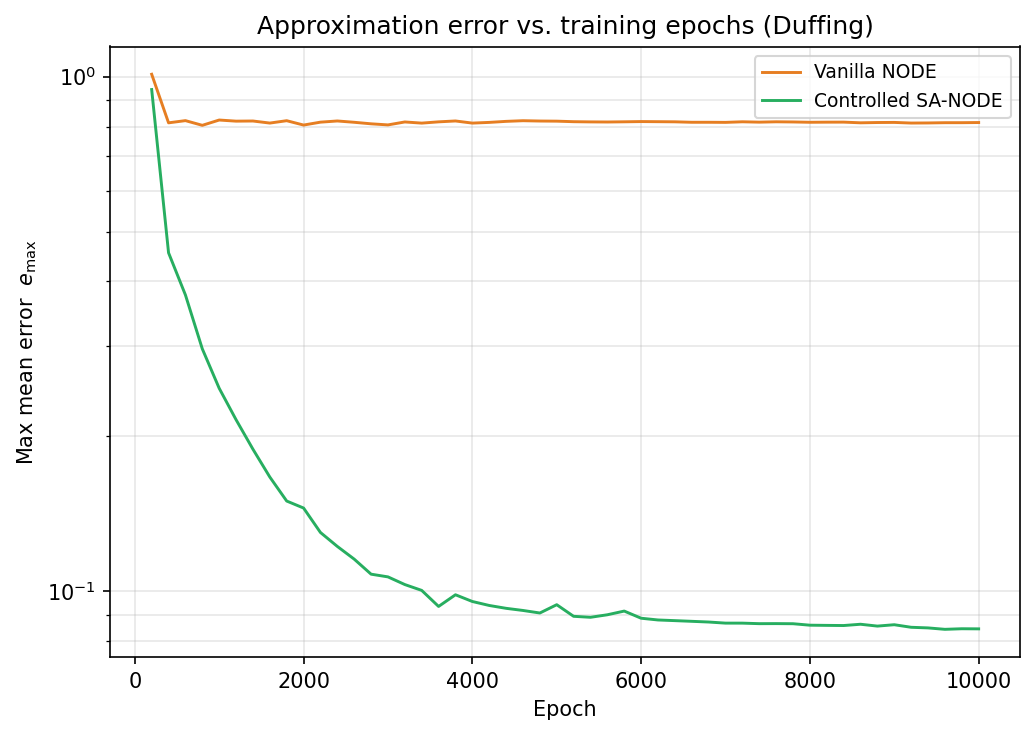}
		\caption{Duffing: $e_{\max}$ vs.\ epochs}
		\label{fig:duffing_error_epochs}
	\end{subfigure}
	\caption{Validation error $e_{\max}$ tracked over $10{,}000$ training
		epochs for the controlled SA-NODE and vanilla NODE, under the
		randomized control protocol. Both axes are on a logarithmic scale.
		The controlled SA-NODE converges smoothly while the vanilla NODE
		stagnates after an initial transient.}
	\label{fig:error_epochs_comparison}
\end{figure}

Figure~\ref{fig:error_dataset_comparison} highlights model performance
across varying training dataset scales. Under severe data scarcity
constraints (as few as $12$ trajectories), both models exhibit elevated and
highly variable errors, with the controlled SA-NODE at
$e_{\max} = 4.142\pm1.2$ and the vanilla NODE at $e_{\max} = 4.167\pm0.94$
for the pendulum (over $3$ repeats). As the dataset grows, the controlled
SA-NODE demonstrates a consistently steeper decline in error. By $84$
trajectories, the controlled SA-NODE reaches $e_{\max} = 0.417\pm0.15$
for the pendulum, compared to $e_{\max} = 1.340\pm0.24$ for the vanilla
NODE - a factor of more than three improvement. At $121$ trajectories, the
gap narrows only slightly: $e_{\max} = 0.416\pm0.21$ versus
$e_{\max} = 1.115\pm0.22$. The Duffing system tells a similar story, with
the controlled SA-NODE achieving $e_{\max} = 0.828\pm0.051$ at $12$
trajectories compared to $e_{\max} = 2.073\pm0.53$ for the vanilla NODE.
These results demonstrate that the controlled SA-NODE not only converges
faster and to lower errors, but also generalizes more reliably from small
datasets.

\begin{figure}[ht!]
	\centering
	\begin{subfigure}[b]{0.49\textwidth}
		\centering
		\includegraphics[width=\textwidth]{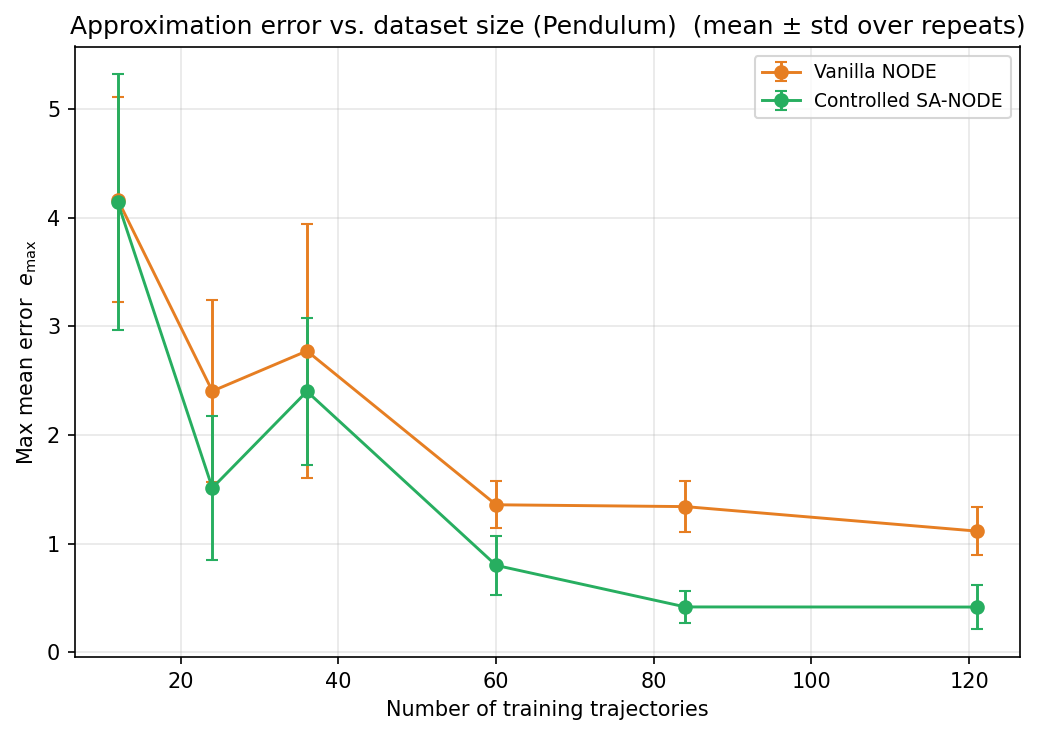}
		\caption{Pendulum: $e_{\max}$ vs.\ dataset size}
		\label{fig:pendulum_error_dataset}
	\end{subfigure}
	\hfill
	\begin{subfigure}[b]{0.49\textwidth}
		\centering
		\includegraphics[width=\textwidth]{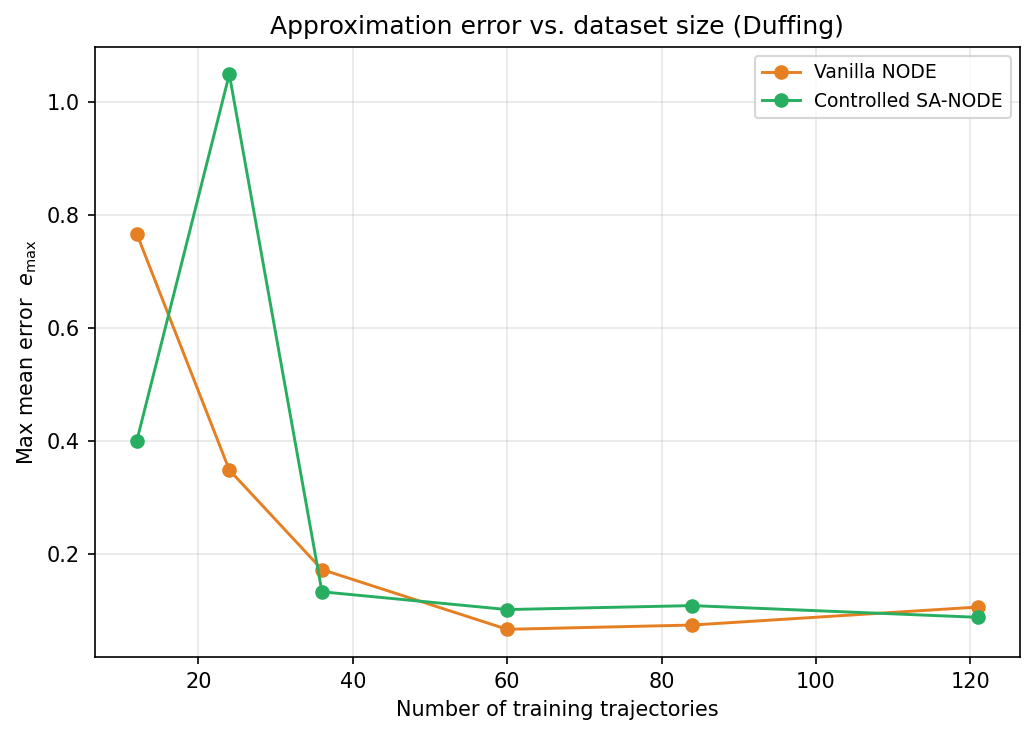}
		\caption{Duffing: $e_{\max}$ vs.\ dataset size}
		\label{fig:duffing_error_dataset}
	\end{subfigure}
	\caption{Generalization capability and data efficiency of the controlled
		SA-NODE and vanilla NODE evaluated against varying numbers of
		training trajectories, under the randomized control protocol. Points
		show the mean $e_{\max}$ and error bars show $\pm1$ standard
		deviation over $3$ independent repeats.}
	\label{fig:error_dataset_comparison}
\end{figure}

\subsection{Approximate controllability experiments}

We finally investigate the approximate controllability properties of the
trained controlled SA-NODE framework across both validation domains. Given
an initial condition $z_0\in\mathbb{R}^d$ and a target state
$z_T^{\mathrm{target}}\in\mathbb{R}^d$, the objective is to optimize a
control sequence $u\in L^\infty(0,T)$, not restricted to the sinusoidal
family \eqref{eq:randomcontrol} used for training, such that
\begin{equation}
	\label{eq:approximatecontrol}
	\|y(T)-z_T^{\mathrm{target}}\| < \varepsilon.
\end{equation}
This control optimization is performed on the SA-NODE model trained under
the randomized-control protocol of Section~\ref{subsec:setup}, and so the
resulting controllability test also serves as a check that $f_\Theta$ and
$G_\Phi$ were learned as separately meaningful objects: steering with an
\emph{out-of-family} control $u(t)$ exercises $G_\Phi$ in a way that the
earlier training procedure could not pre-determine.

The control functions are evaluated over $1000$ optimization iterations
using a predefined functional basis setup. For the controlled pendulum
configuration, the terminal controllability error converges as follows:
\begin{align*}
	\text{Step } 100\text{:} \quad e_T &= 1.1671\times10^{-3}, \\
	\text{Step } 500\text{:} \quad e_T &= 1.9799\times10^{-3}, \\
	\text{Step } 1000\text{:} \quad e_T &= 1.6974\times10^{-3}.
\end{align*}
As visible in Figure~\ref{fig:pendulum_controllability}, the optimization
converges rapidly in the first $\sim100$ steps, undergoes brief oscillations
as the optimizer refines the control profile, and then settles into a smooth
descent, achieving a final terminal error of $e_T = 1.6974\times10^{-3}$.
The left panel confirms that the optimized control successfully steers the
trajectory from the initial state (green dot at approximately $(0.75, 0.40)$)
to within a small neighborhood of the target (star at approximately
$(-0.65, 1.00)$), with the endpoint (red dot) nearly coinciding with the
target.

The stable numerical gradients provided by the model infrastructure allow a
similar tracking execution on the highly nonlinear Duffing oscillator
landscape. The terminal controllability error for the Duffing system
converges as follows:
\begin{align*}
	\text{Step } 100\text{:} \quad e_T &= [\text{TBD}], \\
	\text{Step } 500\text{:} \quad e_T &= [\text{TBD}], \\
	\text{Step } 1000\text{:} \quad e_T &= [\text{TBD}].
\end{align*}
The corresponding steering trajectories and terminal convergence
profiles for both benchmark systems are illustrated in
Figures~\ref{fig:pendulum_controllability}
and~\ref{fig:duffing_controllability}, respectively.

\begin{figure}[ht!]
	\centering
	\includegraphics[width=\textwidth]{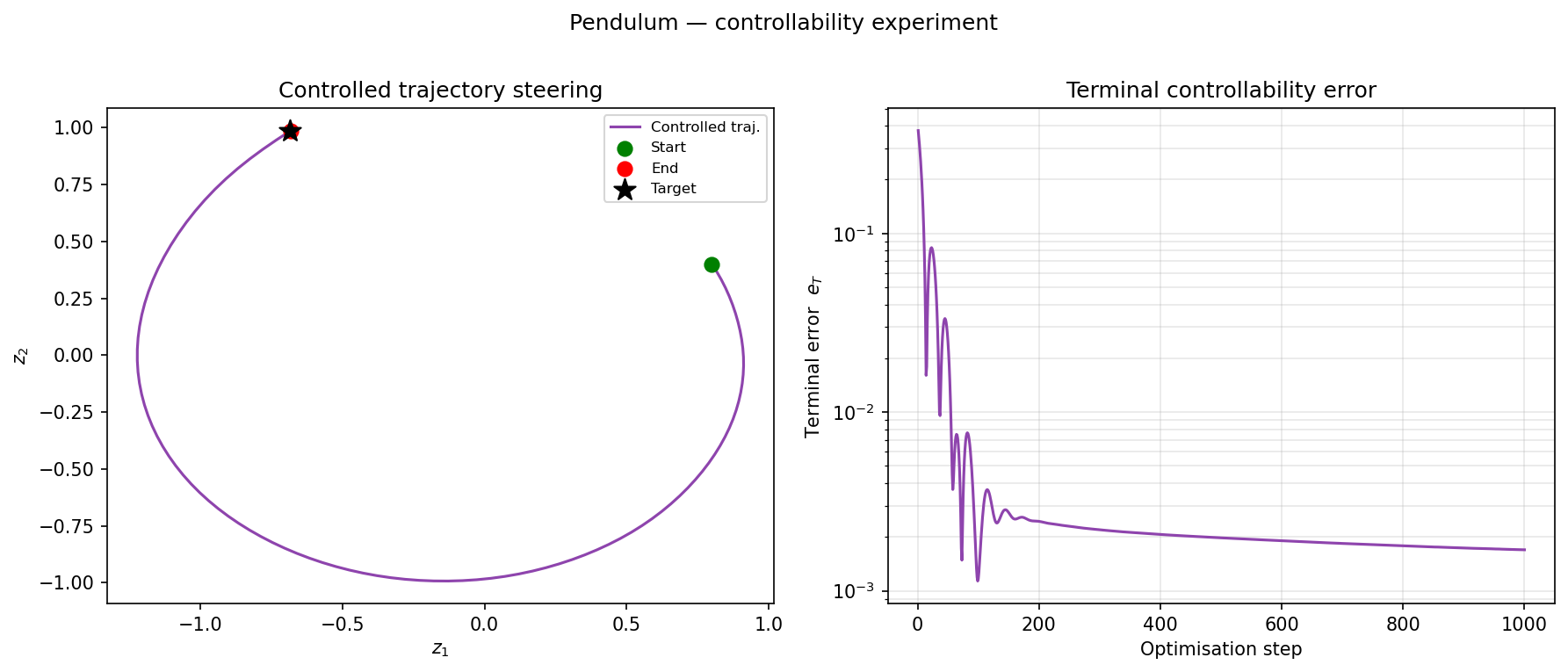}
	\caption{
		Approximate controllability experiment for the controlled SA-NODE
		(trained under the randomized control protocol) on the pendulum
		system. Left: controlled trajectory steering from the initial state
		(green dot) to the target (star), with the achieved endpoint shown
		as a red dot. Right: terminal controllability error $e_T$ over
		$1000$ optimisation steps, with a final value of
		$e_T = 1.6974\times10^{-3}$.
	}
	\label{fig:pendulum_controllability}
\end{figure}

\begin{figure}[ht!]
	\centering
	\includegraphics[width=\textwidth]{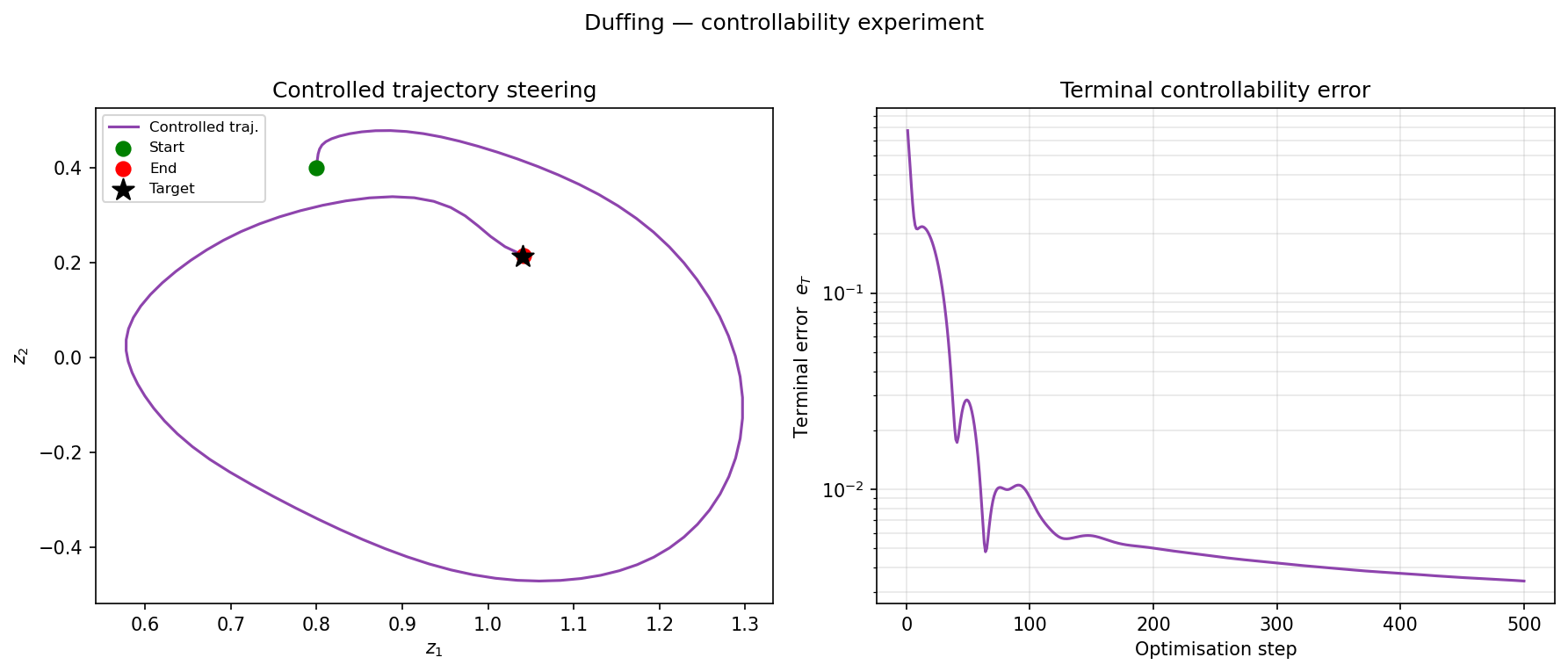}
	\caption{
		Approximate controllability experiment for the controlled SA-NODE
		(trained under the randomized control protocol) on the Duffing
		oscillator. Left: controlled trajectory steering from the initial
		state (green dot) to the target (star), with the achieved endpoint
		shown as a red dot. Right: terminal controllability error $e_T$
		over $1000$ optimisation steps.
	}
	\label{fig:duffing_controllability}
\end{figure}

The learned control profiles demonstrate smooth temporal behaviors across
both benchmark systems, solidifying the applicability of controlled SA-NODEs
for learning-based stabilization, optimal control tracking, and scientific
machine learning framework implementations.

\section{Conclusion}
 \label{sec:conclusion}

In this paper, we introduced a new class of controlled semiautonomous neural ordinary differential equations (controlled SA-NODEs) for nonlinear control-affine dynamical systems. The proposed framework extends semiautonomous neural ODE architectures by incorporating external control actions while preserving reduced parameter complexity through time-independent trainable coefficients. We established universal approximation results showing that controlled SA-NODEs approximate trajectories of nonlinear controlled systems uniformly on compact sets of initial conditions and admissible controls. Under additional Sobolev and Barron regularity assumptions, quantitative approximation estimates were derived with convergence rates of order
\[
\mathcal{O}\left(P^{-1/2}+Q^{-1/2}\right).
\]
We further proved that approximate controllability properties of the original system are preserved under the controlled SA-NODE approximation on compact subsets. Numerical experiments on controlled pendulum and Duffing oscillator systems demonstrated accurate trajectory reconstruction and effective controllability performance with significantly fewer trainable parameters than classical NODE architectures. Future research directions include extensions to stochastic systems, partial differential equations, optimal control problems, and feedback stabilization frameworks.

\bibliographystyle{unsrtnat} 
\bibliography{references} 

\end{document}